\documentclass[10pt,a4paper,reqno]{amsart}
\usepackage{amsthm,amssymb,amsmath, enumerate}
\usepackage{mathrsfs,setspace,pstricks,multicol,multirow, latexsym}
\usepackage{epsfig, subcaption, graphics,graphicx}
\usepackage{dsfont,float}
\usepackage{hyperref}
\hypersetup{
	colorlinks,
	citecolor=blue,
	filecolor=black,
	linkcolor=red,
	urlcolor=black
}
\usepackage{xcolor}
\usepackage{bbm}
\advance\oddsidemargin by -1.6cm
\advance\evensidemargin by -1.6cm

\newcommand\dela[1]{}
\numberwithin{equation}{section}

\newtheorem{thm}{Theorem}[section]

\newtheorem{lem}[thm]{Lemma}
\newtheorem{pr}[thm]{Proposition}
\newtheorem{rmk}[thm]{Remark}

\newcommand{\noi}{\noindent}
\newcommand{\cal}{\mathcal}
\allowdisplaybreaks

\date{}

\begin{document}
	\pagenumbering{arabic}
	
	\author{Anindya Goswami}
	\address{Department of Mathematics, IISER Pune, India}
	\email{anindya@iiserpune.ac.in}
	
	\author{Nimit Rana}
	\address{Fakult\"at f\"ur Mathematik, Universit\"at Bielefeld,  Germany}
	\email{nrana@math.uni-bielefeld.de}
	
	\author{Tak Kuen Siu}
	\address{Department of Actuarial Studies and Business Analytics, Macquarie Business School, Macquarie University, Sydney, NSW, Australia}
	\email{ken.siu@mq.edu.au} 
		
	
	\title[Regime Switching Optimal Growth Model]{Regime Switching Optimal Growth Model with Risk Sensitive Preferences}\thanks{This research was supported in part by the SERB MATRICS (MTR/2017/000543), DST FIST (SR/FST/MSI-105) and NBHM 02011/1/2019/NBHM(RP)R\& D-II/585. The second author would like to acknowledge the German Science Foundation DFG to provide the financial support, through the Research Unit FOR 2402,  during the work of this manuscript.}
	
	\addtocounter{footnote}{-1} \vskip 1 true cm
\begin{abstract}
We consider a risk-sensitive optimization of consumption-utility on infinite time horizon where the one-period investment gain depends on an underlying economic state whose evolution over time is assumed to be described by a discrete-time, finite-state, Markov chain. We suppose that the production function also depends on a sequence of i.i.d. random shocks. For the sake of generality, the utility and the production functions are allowed to be unbounded from above. Under the Markov regime-switching model, it is shown that the value function of optimization problem satisfies an optimality equation and that the optimality equation has a unique solution in a particular class of functions. Furthermore, we show that an optimal policy exists in the class of stationary policies. We also derive the Euler equation of optimal consumption. Furthermore, the existence of the unique joint stationary distribution of the optimal growth process and the underlying regime process is examined. Finally, we present a numerical solution by considering power utility and some hypothetical values of parameters in a regime switching extension of Cobb–Douglas production rate function.
\end{abstract}

\maketitle
	
\noi {\bf Key words: } Regime switching models, Growth models, Risk sensitive Preferences, Optimal consumption, Euler equation\\
\noi {\bf AMSC: } 91B62, 91B55, 91B70, 60J10, 90C40
\section{Introduction}

\noindent Stochastic optimal growth model pioneered by \cite{bm}, has received considerable attention in the economics or related literature. We consider, in this paper, a regime switching optimal growth model for a single sector in the discrete time settings where the agent has risk sensitive preferences. Markov regime-switching models have important applications in economics and econometrics (see, for example, \cite{ham1989,ham2016}). It is assumed that at the beginning of a trading day an economic agent can access to information about the income from the previous investment and the current economic state. Given this information, the agent invests a portion of the  previous income in the production technology and consumes the rest. Depending on the investment and the economic state, the next random income is realized at the beginning of the next day. The next market state is also observed simultaneously. This cycle goes on. The agent's objective is to maximize the discounted risk sensitive non-expected utility of consumption in the infinite time horizon. \\

\noindent A similar problem has been studied in \cite{bjjet} assuming absence of transitions in the economic state. In other words the productivity shocks have been assumed to be i.i.d., instead of following Markovian dynamics. We extend the results of \cite{bjjet} to include the regime switching scenarios, since many state of the art macroeconomic models consider Markov shocks in productivity rate. For example, \cite{m}  considered a real business cycles (RBC) model, where technological disturbances including interest-rate and productivity shocks were assumed to be governed by two-state Markov chain, (see Page 802 therein). \cite{s1994} discussed a RBC theory in which productivity shocks were supposed to follow a stationary Markov process, (see Page 1754 therein). \cite{vv} assumed that the technology shock follows a two-state Markov switching process when discussing asymmetries in a RBC model. \cite{h} adopts a Markov process for modelling persistency in productivity shocks when investigating optimal environmental policies responding to business cycles. \cite{at} supposed that an aggregate productivity shock was governed by a first-order Markov process when studying polarized business cycles. The Markov regime-switching model for optimal growth considered here can capture an important aspect of economic growth, namely the impacts of transitions in different phases of business cycles such as expansion and recession on the growth in an unified setting. Specifically, the model may provide theoretical insights into making investment-consumption decisions with the objective of achieving sustainability in the economic growth in the presence of changing economic regimes. We describe probabilistic transitions in economic regimes using a discrete-time, finite-state, Markov chain whose probability laws are given. To the best of our knowledge, this problem has not been studied in the literature.\\

\noindent For the sake of generality, in this paper, the production function and the utility function are allowed to be unbounded from above. We refer to \cite{bjjet}, \cite{d2}, \cite{kami} and \cite{w} and references therein for similar considerations. We follow the weighted supremum norm approach for studying the optimization problem. As in \cite{hs} and \cite{bjjet} the finite horizon objective function has been defined recursively. Its limiting value is then shown to exist which gives rise to the non-expected discounted utility in the infinite time horizon. We then study the corresponding optimality equation using the Banach fixed point theorem. To show that the dynamic programming operator maps a space of certain functions into itself, we consider the class of  concave, non-decreasing and non-negative functions with weighted supremum norm as in \cite{bjjet}. In \cite{bjjet} the contraction property of the operator has been established by applying Chebyshev's association inequality of random variables. We have first extended that association inequality to the regime switching case (see Lemma \ref{lem-covl}) and then  applied the inequality for establishing the contraction mapping. Furthermore, it is shown that the optimal investment policy exists in the class of stationary policies.\\

\noindent The Euler equation for optimal consumption is also considered. An Euler equation generally states that the average gain in utility for saving now and then consuming in future, instead of immediate consumption should, after discounting, be equal to the utility gain of consuming now. The readers can find the study of Euler equation in \cite{bm} and \cite{kami}. It is interesting to note that under the present settings for a nonzero risk aversion parameter, the Euler equation involves an expectation with respect to a probability measure different from the standard equilibrium market measure. Specifically, under an optimal consumption policy, the measure depends on the value function. Consequently, higher (lower) probabilities are assigned to the scenarios corresponding to smaller (larger) values of the value function. In view of this, such measures are also referred to as the worst-case measure (see \cite{hs2007}).\\

\noindent In addition to the above results, we have also established that, under additional assumptions, the joint process consisting the optimal income and the underlying economic state has a non-trivial joint stationary distribution. This is accomplished as an application of a result from \cite{mt}. To facilitate the application of this result, we prove that the stochastic kernel of the joint process is Feller and bounded in probability. \\

\noindent Numerical results based on hypothetical parameter values in a parametric model with three regimes of the production rate are provided to illustrate the impacts of changes in the production rate regime on the optimal investment-income ratio and the value function. To this end, a regime switching extension of the Cobb–Douglas production function and the power utility is considered. The numerical results are presented graphically along with some insightful economical interpretations. In this connection a fair comparison is added between this and the fixed regime counter part examined in \cite[Example 1]{bjjet}. Prior to choosing specific numerical values for computation, we identify  the range of parameter values for which all the theoretical assumptions in this paper are fulfilled.\\

\noindent The rest of this paper is organized as below. In Section 2 we present the regime switching investment-consumption model for a single sector. The optimization problem is formulated in Section 3. In Section 4 the optimality equation of the value function is obtained. Moreover, the existence of the optimal stationary policy is also established. The Euler equation for optimal consumption is derived in Section 5. The establishment of the existence of the joint stationary distribution for the optimal income process and the underlying economic state is presented in Section 6. The numerical results and their interpretations for a natural regime switching extension of a model with Cobb–Douglas production function, coupled with a power utility are discussed in Section 7. We conclude the paper in Section 8 which contains some directions of future investigations. For the sake of self-containedness some results with proofs have been added at the end in the Appendix section.

\section{Multiperiod Growth and Consumption Model}
\noindent In this section, a Markov regime-switching extension to the stochastic optimal growth model with risk sensitive preferences in \cite{bjjet} is provided.
Firstly, some notations are defined.
$\mathbb{N}$ denote the set of positive integers. A complete probability space $(\Omega, {\cal F}, \mathbb{P})$ is considered on which all the random processes are defined. It is supposed that at each time epoch $k \in \mathbb{N}$, an economic agent has an income of the amount $x_k$ and wishes to allocate the income between consumption and investment. Let $a_k$ and $y_k$ denote the consumption and investment at the time epoch $k$, respectively. Assume that the investment $y_k$ is used as an input for production and that this input is a key factor influencing the output at the next time epoch $k+1$, which gives rise to the income of the agent $x_{k+1}$. Furthermore, it is supposed that the output $x_{k+1}$ is also influenced by the observable market regime at the time epoch $k$, which is described by $\theta_k$. To take account of random perturbation, it is assumed that the income dynamics are governed by the following regime dependent production equation:
$$ x_{k+1}=f(\theta_k, y_k,\xi_k).$$
Here $(\xi_k)_{k\in\mathbb{N}}$
is a sequence of independent and identically distributed (i.i.d.) random variables taking their values in $[0,\infty)$ with common distribution $\nu$. Note that $\xi_k$ represents the random shock at the time epoch $k$. $\Theta := \{ \theta_k \}_{k \in \mathbb{N}}$ is a discrete-time time-homogeneous Markov chain with a finite state space ${\cal S}$. The probability laws of the chain are described by a given transition probability matrix $p
:= (p_{\theta \theta'})_{{\cal S}\times {\cal S}}$. The initial wealth
$x_1$ is assumed to be drawn from a given  distribution $p^0$ on $[0,\infty)$. and
$$f:\mathcal{S}\times [0,\infty) \times [0,\infty) \to [0,\infty),$$
is a Markov regime-switching production function. We assume that the regime-switching production function possesses the following property.
\begin{enumerate}
  \item[(F1)] For every $z\ge 0$ and for each $\theta\in \mathcal{S}$  the function $f(\theta,\cdot,z):[0,\infty) \to [0,\infty)$ is non-decreasing,
  continuous, concave,  and for every $y\ge 0$ the function
  $f(\theta,y,\cdot):[0,\infty) \to [0,\infty)$ is Borel measurable.
\end{enumerate}

\noindent Note that the above assumption on the production function is reasonable because the non-decreasing property of the production function is due to the fact that the output does not decrease as the input increases; the concavity property of the production function is due to the diminishing marginal productivity; the Borel measurability of the production function is a technical property so that the expectation of the production function with respect to the random shock is well-defined.\\

\noindent Let $D$ denote the following closed lower triangle
$$D:=\{(x,y):\; x\in[0,\infty),\
y\in[0,x]\}.$$
Then the history process $h_k$ of the income-regime-investment till time $k\in\mathbb{N}$ is given by
 $$h_k=\left\{\begin{array}{l@{\quad \quad}l}
(x_1,\theta_1),& \mbox{ for } k=1,\\
 (x_1, \theta_1,y_1, x_2,\theta_2,y_2,\ldots, x_{k-1},  \theta_{k-1},y_{k-1}, x_k,\theta_k), & \mbox{ for } k\ge2,
 \end{array}\right.$$
\noindent where $(x_i,y_i)\in D$  for all  $i\in \mathbb{N}$. The last constraint is due to the fact that every time, only the income is invested after subtracting a non-negative amount for consumption. Hence, the production input $y_k$ at the time epoch $k$ is constrained by the income $x_k$ at that time epoch. Let the set $H_k\subset [0,\infty)\times \mathcal{S} \times ( [0,\infty)\times [0,\infty)\times \mathcal{S} )^{k-1}$ denote the set of all feasible histories till time $k\ge 1$. The Borel sigma field on $H_k$ be denoted by $\mathcal{H}_k$. An  \emph{ investment policy} $\pi$ is a sequence $(\pi_k)_{k\in \mathbb{N}},$ where for each $k\in \mathbb{N}$ $\pi_k:H_k\to [0,\infty)$ is
a $\mathcal{H}_k$-measurable mapping so that $(x_k, \pi_k(h_k))\in D$.
The collection of all investment policies is denoted by $\Pi$. A policy $\pi$ is said to be Markov if there exists a sequence of Borel measurable  maps $\phi_k$ in
$$\Phi:=\{\phi: [0,\infty)\times \mathcal{S} \to [0,\infty) \textrm{ Borel measurable}\mid (x, \phi(x,\theta))\in D\; \forall x\in [0,\infty), \theta\in \mathcal{S} \}$$
such that $\pi_k(h_k) =\phi_k(x_k,\theta_k)$ for every $k\in \mathbb{N}$. A Markov policy $\{\phi_k\}_k$ is said to be stationary if there exists a fixed $\phi\in \Phi$, so that $\phi_k=\phi$ for all $k\in \mathbb{N}$.
Conventionally, we identify a stationary policy with the mapping $\phi\in \Phi$. Thus the set of all stationary investment policies is also denoted by $\Phi(\subset \Pi)$.\\

\noindent The utility function $u:[0,\infty)\to [0,\infty)$ is assumed to be
\begin{enumerate}
    \item[(U1)] strictly concave, increasing, continuous at zero and
  $u(0)=0.$
\end{enumerate}
In addition to above, we assume that there exists a continuous function $w: [0,\infty) \times \mathcal{S}  \to  [1,\infty)$ which is non-decreasing in the first variable for each $\theta\in \mathcal{S}$ such that,
\begin{enumerate}
  \item[(U2)] for a fixed  constant $d>0$
  $$u(a)\le dw(a,\theta),\quad \mbox{for all } \theta \in \mathcal{S}, a\in [0,\infty);$$
\item[(F2)]  for a fixed  constant $\alpha\in (0, 1/\beta)$,
\begin{equation}
\nonumber\label{F2}
\sup_{y\in [0,x]} \sup_{\theta'\in\mathcal{S}} \int_{[0,\infty)} w(f(\theta,y,z),\theta')\nu(dz)\le \alpha w(x,\theta) \quad \mbox{for all } \theta \in \mathcal{S}, x\in [0,\infty),
  \end{equation}
\noindent where $\beta\in (0,1)$ denotes the one-period discounting factor.
\end{enumerate}
\begin{rmk}
If $u$ is a bounded utility, then (U2) and (F2) are trivially true with a constant function  $w = \|u\|_{\sup_{}}$, and  $d=\alpha=1$. However, for unbounded utility, (F2) puts a constraint on $f$. For illustrating the feasibility of (U2)-(F2) in the unbounded utility case, we consider the following toy example. Let $u(a)=\ln(1+a)$, and $f(\theta,y,z)=y$ for all $x,y,z\ge 0, \theta \in \mathcal{S}$. Then (U2)-(F2) hold with $w(x,\theta)=1+\ln(1+x)$, and $d=\alpha=1<1/\beta$ for any  $\beta\in (0,1)$.
\end{rmk}

\noindent Before ending this section we introduce few more notations which are essential for the subsequent sections. Often we write $w_\theta(x)$ instead of $w(x,\theta)$, given in (U2)-(F2) for notational convenience. We define for each $k\in \mathbb{N}$, using $w$, the following space of functions
$$
B_w(H_k):=\left\{v: H_k\to  [0,\infty)  \textrm{ Borel measurable }\mid \sup_{h_k\in H_k}  \frac{v(h_k) }{w(x_k,\theta_k)}<\infty
\right\}.
$$
Thus for every $v_k\in B_w(H_k)$, there exists a constant $d_{v_k}\ge 0$ such that
$$v_k(h_k)\le d_{v_k}w(x_k,\theta_k) = d_{v_k}w_{\theta_k}(x_k) \quad \textrm{ for every }  h_k\in H_k.$$

\noindent The $w$-norm of a Borel measurable function $v: [0,\infty) \times \mathcal{S}  \to {\mathbb R}$ is defined as
\begin{equation}
 \| v\|_w := \sup_{x\in[0,\infty), \theta \in \mathcal{S}}\frac{|v(x,\theta)|}{w(x,\theta)}.
\end{equation}
Then, $(B_w,\|\cdot\|_w)$ is a Banach space (see Proposition 7.2.1 in \cite{hl}), where
$$
B_w =\left\{ v: [0,\infty) \times \mathcal{S} \to {\mathbb R} \textrm{ Borel measurable }\mid \| v\|_w<\infty\right\}.
$$
We further consider a closed subset of $B_w$ by
$$ {\cal B}:= \left\{v \in B_w \mid \mbox{for each $\theta$, } v(\cdot,\theta) \;\mbox{is non-decreasing, non-negative, continuous,
concave} \right\}.$$
Clearly, $({\cal B},\|\cdot\|_w)$ is a complete convex metric space. As in \cite{bjjet}
we study maximization of the agent's non-expected utility defined via the entropic risk measure in the infinite time horizon case. The objective function will be presented in the following section.

\section{The Optimization Problem}\label{sec:optimization}

\noindent For each $h_k\in H_k$ and $\pi= (\pi_k)_{k\in\mathbb{N}} \in \Pi$, let the map
$\rho_{\pi_k,h_k}: B_w(H_{k+1})\to [0,\infty)$
be given by
\begin{align}\label{ece}
 \nonumber   \rho_{\pi_k,h_k}(v_{k+1}) &:= -\frac 1\gamma\ln \left[\sum_{\theta'\in\mathcal{S}} p_{\theta_k \theta'} \int_{[0,\infty)} e^{-\gamma v_{k+1}(h_k,\pi_k(h_k), f(\theta_k,\pi_k(h_k),z),\theta')}\nu(dz) \right]\\
 & =  -\frac{1}{\gamma} \ln E[e^{-\gamma v_{k+1}(h_k,\pi_k(h_k), f(\theta_k,\pi_k(h_k),\xi_k),\theta_{k+1})}\mid h_k].
\end{align}
Some important properties of $\rho_{\pi_k,h_k}$ are listed below whose proofs are deferred to the Appendix.
\begin{thm} \label{P} Let $v, v' \in B_w(H_{k+1})$. For each $h_k\in H_k$ and $\pi= (\pi_k)_{k\in\mathbb{N}} \in \Pi$,
\begin{enumerate}[(i)]
    \item $0= \rho_{\pi_k,h_k}({\bf 0}) \le \rho_{\pi_k,h_k}(v)$ if $v\ge {\bf 0}$, where ${\bf 0}$ is the zero function, i.e., ${\bf 0}(h_k)\equiv 0$ for any $k\ge 1$ and any $h_k\in H_k$;
    \item $\rho_{\pi_k,h_k}$ is monotonic, i.e., if $v\le v'$, then $\rho_{\pi_k,h_k}(v)\le \rho_{\pi_k,h_k}(v')$;
  \item   $\rho_{\pi_k,h_k}(v) \leq E[v(h_k,\pi_k(h_k), f(\theta_k,\pi_k(h_k),\xi_k),\theta_{k+1})\mid h_k]$;
    \item $\rho_{\pi_k,h_k}$ is concave, i.e.,$\rho_{\pi_k,h_k}(\lambda v+(1-\lambda)v')\ge \lambda\rho_{\pi_k,h_k}(v)+(1-\lambda)\rho_{\pi_k,h_k}(v')$
  for any $\lambda\in[0,1]$;
\item under (F2), $\rho_{\pi_k,h_k}(v) \le  d_{v} \alpha w_{\theta_k}(x_k)$ for some  constant $d_{v}\ge 0$;
\item $\rho_{\pi_k,h_k}(\mu v) -\mu \rho_{\pi_k,h_k}(v) \left\{\begin{array}{cc}
\ge 0  & \textrm{ if } \mu \in [0,1] \\
\le 0  & \textrm{ if } \mu \ge 1.
\end{array}\right.$
\end{enumerate}
\end{thm}

\noindent As in \cite{hs} and \cite{bjjet} we model the objectives of an economic agent recursively. In other words, we introduce the objective function using a sequence of operators
$\{L_{\pi_k}\}_{k=1}^\infty$ given by
\begin{align*}
(L_{\pi_k}v_{k+1})(h_k):=&u(x_k-\pi_k(h_k))+\beta\rho_{\pi_k,h_k}(v_{k+1})\\
=& -\frac{\beta}{\gamma} \ln E[e^{-\gamma\left( \frac{1}{\beta}u(x_k-\pi_k(h_k))+v_{k+1}(h_k,\pi_k(h_k), f(\theta_k,\pi_k(h_k),\xi_k),\theta_{k+1}) \right)}\mid h_k]
\end{align*}
where $\beta\in (0,1)$ as described in (F2) and $v_{k+1}\in B_w(H_{k+1})$.

\begin{lem} \label{L} Assume that the conditions (U2) and (F2) hold. Then\\
(i) $L_{\pi_k}:B_w(H_{k+1}) \to B_w(H_{k})$ is well defined, and\\
(ii) $L_{\pi_k}$ is monotone.
\end{lem}
\begin{proof}
The fact that the range of $L_{\pi_k}$ is indeed in $B_w(H_{k})$ can be verified in the following manner. Assuming (U2) and (F2) and using  Theorem \ref{P}(v), we get for every $h_k\in H_k$ and any $k\in\mathbb{N}$, that
\begin{equation}
\label{L_ogr}
0\le (L_{\pi_k}v_{k+1})(h_k)\le d w_{\theta_k}(x_k- \pi_k(h_k))  +\alpha\beta d_{v_{k+1}}w_{\theta_k}(x_k) \le (d+\alpha\beta d_{v_{k+1}})w_{\theta_k}(x_k),
\end{equation}
as $w_\theta$ is non-decreasing. Thus, the map $$L_{\pi_k}: B_w(H_{k+1}) \ni v_{k+1} \mapsto  L_{\pi_k}v_{k+1} \in B_w(H_k),  $$ is well defined.
Moreover, by Theorem \ref{P} (ii) and using $\beta \in (0, 1)$ under (F2), we directly get that
$(L_{\pi_k}v_{k+1})(h_k)\le (L_{\pi_k}{v'}_{k+1})(h_k)$ for $h_k\in H_k$ and $v_{k+1}\le {v'}_{k+1}$, i.e., $L_{\pi_k}$ is monotone.
\end{proof}
\noindent For any initial income $x_1=x$, state $\theta'= \theta$ and  $T\in {\mathbb N}$ we define $T$-stage total discounted utility
\begin{equation}\label{J_T}
J_T(x,\theta,\pi):=(L_{\pi_1}\circ \ldots \circ L_{\pi_T} ){\bf 0}(x,\theta)
\end{equation}
where ${\bf 0}(h_{T+1})=0$ for all $h_{T+1}\in B_w(H_{T+1})$. In particular $J_1(x,\theta,\pi)= L_{\pi_1} {\bf 0}(x,\theta)= u(x-\pi_1(x,\theta)) + \beta\rho_{\pi_1,h_1}({\bf 0})= u(x-\pi_1(x,\theta))$.

\noindent Theorem \ref{thm4} below generalizes some results for performance functionals in Section 2 of \cite{bjjet} to a Markov regime-switching situation, where the performance functionals depend on the modulating Markov chain.
\begin{thm} \label{thm4}
For every $x\in[0,\infty)$, $\theta \in \mathcal{S}$ and $\pi\in\Pi$,\\
(i) the sequence $\{ J_T(x,\theta,\pi)\}_{T\in\mathbb{N}}$ is non-decreasing and is non-negative;\\
(ii) and for each $T\in\mathbb{N}$
$$
J_T(x,\theta, \pi)\le \frac {dw(x,\theta ) }{1-\alpha\beta}.
$$
(iii)
$\lim_{T\to\infty} J_T(x,\theta, \pi)$ exists for every $x\in [0,\infty)$, $\theta \in \mathcal{S}$ and $\pi\in\Pi.$
\end{thm}
\begin{proof}
\noindent The objective function $J_T$, defined as compositions of $\{L_{\pi_k}\}$, is non-negative due to the non-negativity of $L_{\pi_k}$. Using Theorem \ref{P}(i) and the non-negativity of $u$, clearly ${\bf 0}\le L_{\pi_{T+1}}{\bf 0}$. Now by applying $ L_{\pi_1} \circ  \ldots \circ L_{\pi_{T}}$ on both sides and using the monotonic property Lemma \ref{L}(ii), we get $ L_{\pi_1} \circ  \ldots \circ L_{\pi_{T}} {\bf 0}\le L_{\pi_1} \circ  \ldots \circ L_{\pi_{T}}\circ L_{\pi_{T+1}}{\bf 0}$. That is $J_T(x,\pi)\le J_{T+1}(x,\pi)$. Thus (i) is true.\\

\noindent Since $\beta \in (0,1)$ and $\alpha \in (0,1/\beta)$, using (U2) and the non-decreasing property of the utility function we have
\begin{equation}\label{111}
L_{\pi_T}{\bf 0}(h_T)=u(x_T-\pi_T(h_T))\le u(x_T)\le dw_{\theta_T}(x_T)\le \frac{dw_{\theta_T}(x_T)}{1-\alpha\beta},\quad h_T\in H_T.
\end{equation}
Since $L_{\pi_{T-1}}$ is monotonic (Lemma \ref{L}(ii)), we get an inequality by applying $L_{\pi_{T-1}}$ on both sides of \eqref{111}. Then using (\ref{L_ogr}) with $k=T-1$, $v_T(h_T):=dw(x_T)/(1-\alpha\beta)$, and of course $d_{v_T} = \frac{d}{1-\alpha \beta}$, we obtain
\begin{align}
    L_{\pi_{T-1}} (L_{\pi_T} {\bf 0})(h_{T-1})\le  L_{\pi_{T-1}} \left(\frac{dw_{\theta_T}}{1-\alpha\beta}\right)(h_{T-1})\le
dw_{\theta_{T-1}}(x_{T-1})+\alpha\beta\frac{dw_{\theta_{T-1}}(x_{T-1})}{1-\alpha\beta}=\frac{dw_{\theta_{T-1}}(x_{T-1})}{1-\alpha\beta}.
\end{align}
Recall that $w_{\theta_{T-1}}(x_{T-1}) = w(x_{T-1}, \theta_{T-1})$. Continuing this procedure, we finally derive that
\begin{align}
 J_{T}(x,\theta,\pi)&=(L_{\pi_1}\circ \ldots \circ L_{\pi_T} {\bf 0})(x,\theta)
\le \cdots \le \frac{dw(x,\theta)}{1-\alpha \beta}.
\end{align}
From (i) and (ii), $J_T (x, \theta, \pi)$ is non-decreasing and bounded from above and hence $\lim_{T\to\infty} J_T(x,\theta, \pi)$ exists for every $x\in [0,\infty)$, $\theta \in \mathcal{S}$ and $\pi\in\Pi$.

\end{proof}

\noindent
{\it The problem statement.} For an initial income $x\in[0,\infty)$, state $\theta \in \mathcal{S}$ and policy $\pi\in\Pi$ the non-expected discounted utility in the infinite time horizon is given by
\begin{equation}
\label{J}
J(x,\theta,\pi):= \lim_{T\to\infty} J_T(x,\theta,\pi).
\end{equation}
The aim of the economic agent is to find an optimal value (the so-called value function) of
the non-expected discounted utility in the infinite time horizon and a policy  $\pi^* \in \Pi$ for which

$$J(x,\theta, \pi^*)=\sup_{\pi\in\Pi} J(x,\theta,\pi), \quad \mbox{for all } x\in[0,\infty), \theta \in \mathcal{S}.$$

\noindent We conclude this section by the following remarks.
\begin{rmk}
If $\pi= (\pi_k)_{k\in \mathbb{N}} \in \Phi$, then there exists ${\phi}: [0,\infty) \times \mathcal{S} \to [0,\infty)$ such that $\pi_k(h_k) = {\phi}(x_k,\theta_k)$ for all $k \in \mathbb{N}$. Consequently
\begin{eqnarray*}
J_2(x,\theta,\pi)&=&
(L_{\pi_1}\circ L_{\pi_2}){\bf 0}(x,\theta)=L_{\pi_1}( L_{\pi_2}{\bf 0})(x,\theta)\\
&=& u(x-\pi_1(x,\theta))-\frac\beta\gamma\ln\left[ \sum_{\theta'\in\mathcal{S}} p_{\theta \theta'} \int_{[0,\infty)} e^{-\gamma  (L_{\pi_2}{\bf 0})(x,\theta,\pi_1(x,\theta),f(\theta,\pi_1(x,\theta),z), \theta')}\nu(dz) \right]\\
&=&u(x-\pi_1(x,\theta))-\frac\beta\gamma\ln \left[ \sum_{\theta'\in\mathcal{S}} p_{\theta \theta'} \int_{[0,\infty)} e^{-\gamma \left\{ u(f(\theta,\pi_1(x,\theta),z)- \pi_2(x,\theta,\pi_1(x,\theta),f(\theta,\pi_1(x,\theta),z), \theta'))\right\} }\nu(dz) \right]\\
&=&u(x-{\phi}(x,\theta))-\frac\beta\gamma\ln \left[ \sum_{\theta'\in\mathcal{S}} p_{\theta \theta'} \int_{[0,\infty)} e^{-\gamma \left\{ u(f(\theta,{\phi} (x,\theta),z)- {\phi}(f(\theta,{\phi} (x,\theta),z), \theta'))\right\} }\nu(dz) \right]\\
&=&u(x-{\phi}(x,\theta))-\frac\beta\gamma\ln \left[ \sum_{\theta'\in\mathcal{S}} p_{\theta \theta'} \int_{[0,\infty)} e^{-\gamma \left\{ u(f(\theta,{\phi} (x,\theta),z)- \pi_1(f(\theta,{\phi} (x,\theta),z), \theta'))\right\} }\nu(dz) \right]\\
&=& u(x-{\phi}(x,\theta))-\frac\beta\gamma\ln\left[ \sum_{\theta'\in\mathcal{S}} p_{\theta \theta'} \int_{[0,\infty)} e^{-\gamma  J_1(f(\theta,{\phi} (x,\theta),z), \theta',\pi)}\nu(dz) \right].
\end{eqnarray*}
Similarly we can prove that for every $T\ge 1$ and $\pi \in \Phi$
\begin{eqnarray}\label{rem-J2}
J_{T+1}(x,\theta,\pi)&=& u(x-{\phi}(x,\theta))-\frac\beta\gamma\ln\left[ \sum_{\theta'\in\mathcal{S}} p_{\theta \theta'} \int_{[0,\infty)} e^{-\gamma  J_T(f(\theta,{\phi} (x,\theta),z), \theta',\pi)}\nu(dz) \right].
\end{eqnarray}
The above recursion generalizes the recursion for the performance functionals in Eq. (12) of  \cite{bjjet} to a Markov-regime-switching environment.
\end{rmk}

\begin{rmk}
\noindent The fact that the objective function $J_T$ is non-expected utility can be illustrated by further computation of $J_3$. To this end we introduce the notation $P^{\pi}$ to denote the distribution of $\{h_k\}_k$ where $y_k=\pi_k(h_k)$ for all $k\in \mathbb{N}$. Let $E^\pi$ denote the expectation w.r.t. $P^{\pi}$. Using these notations, we write below
\begin{eqnarray*}
J_3(x,\theta,\pi)&=&
(L_{\pi_1}\circ L_{\pi_2}\circ L_{\pi_3}){\bf 0}(x,\theta)=L_{\pi_1}( L_{\pi_2}\circ L_{\pi_3}{\bf 0})(x,\theta)\\
&=& -\frac{\beta}{\gamma} \ln E^{\pi}[e^{-\gamma\left( \frac{1}{\beta}u(x-y_1)+ L_{\pi_2}\circ L_{\pi_3}{\bf 0} (h_{2})\right)}\mid x,\theta]\\
&=& -\frac\beta\gamma\ln E^{\pi} \left[e^{-\gamma \left(\frac{1}{\beta}u(x-y_1)+ u(x_2- y_2) -\frac{\beta}{\gamma} \ln E^{\pi}[e^{-\gamma L_{\pi_3}{\bf 0} (h_{3})}\mid h_2]
\right)}\mid x,\theta \right]\\
&=& -\frac\beta\gamma\ln E^{\pi} \left[e^{-\gamma \left(\frac{1}{\beta}u(x-y_1)+ u(x_2- y_2) \right)} \left(E^{\pi}[e^{-\gamma u(x_3-\pi_3(h_3))}\mid h_2]
\right)^\beta \mid x,\theta \right]\\
&=& -\frac\beta\gamma\ln E^{\pi} \left[ \left(E^{\pi}\left[e^{-\gamma \left(\frac{1}{\beta^2}u(x-y_1)+ \frac{1}{\beta}u(x_2- y_2) + u(x_3-y_3)\right)}\mid h_2\right]
\right)^\beta \Big| x,\theta \right].\end{eqnarray*}
Clearly for $\beta\neq 1$, the above expression cannot be written as a conditional expectation of an utility function.
\end{rmk}

\section{Value Function and Optimal Stationary Policy}
\noindent This section is dedicated for establishing the following theorem. This asserts that the non-expected discounted utility in the infinite time horizon (as defined in earlier section) can be maximized by a stationary investment policy. Moreover, the optimal utility value function and the optimal stationary policy can be obtained by solving a Bellman equation.
\begin{thm}\label{mainthm}
 Assume (U1)-(U2) and (F1)-(F2). Then, the following statements are true.
\begin{itemize}
\item[(a)] There exists a unique function $V\in {\cal B}$ such that
  \begin{eqnarray}
  \label{oe}
  V(x,\theta )&=&\sup_{y\in[0,x]}\left(u(x-y)-
\frac\beta\gamma\ln\left[\sum_{\theta'\in\mathcal{S}} p_{\theta \theta'} \int_{[0,\infty)} e^{-\gamma V(f(\theta,y,z),\theta')}\nu(dz) \right] \right)
\end{eqnarray}
for each $x\in [0,\infty)$ and $\theta \in \mathcal{S}$. Moreover, $V(\cdot,\theta)$ is strictly concave for each $\theta$.
\item[(b)] Let $V\in {\cal B}$ be as in \eqref{oe}, then $\widehat{V}: [0,\infty) \times \mathcal{S} \to {\mathbb R}$, given by
\begin{equation}
\label{b}
\widehat{V}(y,\theta):= -\frac 1\gamma\ln\sum_{\theta'\in \mathcal{S}} p_{\theta \theta'} \int_{[0,\infty)} e^{-\gamma V(f(\theta,y,z),\theta')}\nu(dz),
\end{equation}
is strictly concave, continuous and non-decreasing. There exists a unique ${\phi}^*\in \Phi$ such that
\begin{eqnarray}\label{oe_max}
V(x,\theta)&=& u(x-{\phi}^* (x,\theta))-\frac\beta\gamma\ln \sum_{\theta'\in\mathcal{S}} p_{\theta \theta'}\int_{[0,\infty)}e^{-\gamma V(f(\theta,{\phi}^*(x,\theta),z),\theta')}\nu(dz).
\end{eqnarray}
Moreover, the functions $x\mapsto {\phi}^*(x,\theta)$ and
  $x\mapsto c^*(x,\theta):=x-{\phi}^*(x,\theta)$ are continuous and non-decreasing for each $\theta$.
   \item[(c)] $V(x,\theta)=\sup_{\pi\in \Pi} J(x,\theta,\pi)=J(x,\theta, {\phi}^*)$ for all $x\in [0,\infty)$ and for all $\theta \in {\cal S}$, i.e.\ there exists an optimal stationary policy ${\phi}^*$.
\end{itemize}
\end{thm}

\noindent Throughout this section we assume that (U1)-(U2) and (F1)-(F2) are satisfied.
We start with a result that we shall use in many places. Before giving a proof for Theorem \ref{mainthm}, several results are presented in the following lemmas and their proofs are provided.

\begin{lem} \label{lem-basic}
Let $v\in {\cal B}$ and $\theta \in \mathcal{S}$.  Then, the function
$$y\mapsto \widehat{v}(y,\theta):= -\frac 1\gamma\ln \left[\sum_{\theta'\in\mathcal{S}} p_{\theta \theta'} \int_{[0,\infty)} e^{-\gamma v(f(\theta,y,z),\theta')}\nu(dz) \right] $$
is continuous, concave, non-decreasing and non-negative. Additionally if $v$ is strictly concave, $\widehat{v}$ is also strictly concave.
\end{lem}
\begin{proof}
Let us fix $v\in {\cal B}$ and $\theta \in \mathcal{S}$. By (F1), for every $y,z \in [0,\infty)$, we have $0\le f(\theta,y,z)$. Since, for each $\theta' \in \mathcal{S}$, $v(\cdot,\theta')$ is non-negative and non-decreasing,
$$0\le v(0,\theta')\le v(f(\theta,y,z),\theta') \textrm{ for any } y,z \in[0,\infty).$$
Consequently, since $(\mathbb{R}_+, {\cal B} (\mathbb{R}_+), \nu)$, where ${\cal B} (\mathbb{R}_+)$ is the Borel
$\sigma$-field on $\mathbb{R}_+$ is a probability space, \begin{align*}
    \int_{[0,\infty)} e^{-\gamma v(f(\theta,y,z),\theta')}\nu(dz) \le 1.
\end{align*}
This implies that $0\le \widehat{v}$ as $\widehat{v}$ involves composition of negative log function of an average of the above expression.\\

\noindent Furthermore, since by (F1),  for every  $z\in[0,\infty)$, the function $f(\theta,\cdot,z)$ is non-decreasing  and, by definition of $\mathcal{B}$, $v$ is non-decreasing, for every  $\theta' \in \mathcal{S}$, the composition $v(f(\theta,\cdot,z),\theta')$ is non-decreasing. Again as $y\mapsto e^{-y}$ is decreasing,  $y\mapsto \sum_{\theta'\in\mathcal{S}} p_{\theta \theta'} \int_{[0,\infty)} e^{-\gamma v(f(\theta,y,z), \theta')}\nu(dz)$ is an average of a family of decreasing functions. Finally as $y\mapsto -\ln y$ is also decreasing,  $\widehat{v}$, the composition of the two above decreasing functions is non-decreasing.\\

\noindent Since composition of continuous functions is continuous, by virtue of (F1), $e^{-\gamma v( f(\theta,\cdot,z),\theta')}$ is continuous for every $z\in[0,\infty)$ and $\theta, \theta' \in \mathcal{S}$. In addition to that $0< e^{-\gamma v( f(\theta,\cdot,\cdot),\cdot)}<1$. Hence, the dominated convergence theorem implies that $\widehat{v}$ is continuous.\\

\noindent In order to show the concavity of $\widehat{v}$ w.r.t. $y$, let $y=\lambda y'+ (1-\lambda) y'',$
where $\lambda\in(0,1)$. By (F1) for each $z\in[0,\infty)$ and $\theta\in\mathcal{S}$ $$f(\theta,y,z)\ge \lambda f(\theta,y',z)+ (1-\lambda) f(\theta,y'',z)$$
holds. Since, $v(\cdot, \theta)$ is non-decreasing and concave in its first argument for each $\theta \in {\cal S}$, we obtain, for each $\theta, \theta'\in\mathcal{S}$, and $z\in[0,\infty)$
$$v(f(\theta,y,z),\theta')\ge v(\lambda f(\theta,y',z)+ (1-\lambda) f(\theta,y'',z),\theta')\ge
\lambda v(f(\theta,y',z),\theta')+ (1-\lambda) v(f(\theta,y'',z),\theta').$$
Applying properties (ii) and (iv) in Theorem \ref{P} respectively gives
\begin{align}\label{oper_rho}
\widehat{v}(y,\theta)& \ge -\frac 1\gamma
  \ln \left[\sum_{\theta'\in\mathcal{S}} p_{\theta \theta'} \int_{[0,\infty)} e^{-\gamma \left(\lambda v(f(\theta,y',z),\theta')+ (1-\lambda) v(f(\theta,y'',z),\theta') \right) }\nu(dz) \right] \nonumber \\
  &
  \ge
  \lambda \widehat{v}(y',\theta)+ (1-\lambda) \widehat{v}(y'',\theta).
\end{align}
Hence $\widehat{v}$ is concave in its first argument for
each $\theta \in {\cal S}$. Again, if $v$ is strictly concave, the inequality above \eqref{oper_rho} is also strict, implying the strict concavity of $\widehat{v}$.
\end{proof}

\begin{lem} \label{lem-incr}
Assume that $\phi\in\Phi(\subset \Pi)$ is a non-decreasing function in its first argument for each $\theta \in {\cal S}$ such that $x\mapsto x-{\phi}(x,\theta)$ is also non-decreasing for each $\theta\in\mathcal{S}$.
Then, for any $T\in\mathbb{N}$
the function $x\mapsto J_T(x,\theta,\phi)$ is non-decreasing and continuous.
\end{lem}
\begin{proof} First we note that $\phi(\cdot, \theta)$ is Lipschitz continuous uniformly in $\theta \in {\cal S}$ with a common constant 1. Indeed, as $x\mapsto x-{\phi}(x,\theta)$ is non-decreasing, for $x \leq y$,
\begin{align*}
    x-{\phi}(x,\theta) \leq y - {\phi}(y,\theta), \quad\textrm{ i.e., }  {\phi}(y,\theta) - {\phi}(x,\theta) \leq y-x .
\end{align*}
Next, since $\phi$ is non-decreasing, for $x \leq y$,
\begin{align*}
    {\phi}(x,\theta) \leq {\phi}(y,\theta), \quad \textrm{ i.e., } 0 \leq {\phi}(y,\theta) - {\phi}(x,\theta).
\end{align*}
Thus $0 \leq {\phi}(y,\theta) - {\phi}(x,\theta)\le y-x$. Hence $\phi(\cdot, \theta)$ is  Lipschitz continuous with constant $1$ for all $\theta \in {\cal S}$. \\

\noindent Next we proceed by induction. For $T=1$ the assertion is true by (U1) since
\begin{align*}
    J_1(x,\theta,\phi) = u(x-{\phi}(x,\theta)),
\end{align*}
and $x-\phi(x,\theta)$ is non-decreasing in $x$.
Assume that it holds for some $T\in \mathbb{N}.$ From \eqref{rem-J2} recall that,
 \begin{eqnarray*}
J_{T+1}(x,\theta,\phi)&=& u(x-{\phi}(x,\theta))-\frac\beta\gamma\ln\left[ \sum_{\theta'\in\mathcal{S}} p_{\theta \theta'} \int_{[0,\infty)} e^{-\gamma  J_T(f(\theta,\phi(x,\theta),z), \theta',\phi)}\nu(dz) \right].
\end{eqnarray*}
Now as in Lemma  \ref{lem-basic}, by changing the role of $v$ to $J_T$, and using (U1), the conclusion follows.
\end{proof}

\begin{lem} \label{lem-covl}
For each $j\in \{1,2\}$, let $[0,\infty)\times\mathcal{S} \ni (x,\theta') \mapsto g_j(x,\theta') \in [0,\infty]$, be non-decreasing in $x$. Furthermore, assume that one of $g_1$ and $g_2$ is constant w.r.t. $\theta'$. Then
\begin{eqnarray*}
\lefteqn{
-\frac1\gamma\ln \left[ \sum_{\theta'\in\mathcal{S}} p_{\theta \theta'} \int_{[0,\infty)}e^{-\gamma\left( g_1(f(\theta,y,z),\theta')+ g_2(f(\theta,y,z),\theta')\right)}\nu(dz) \right]  }\\
&\le& -\frac1\gamma\ln \left[ \sum_{\theta'\in\mathcal{S}} p_{\theta \theta'} \int_{[0,\infty)}e^{-\gamma\left( g_1(f(\theta,y,z),\theta') \right)}\nu(dz) \right]
-\frac1\gamma\ln \left[ \sum_{\theta'\in\mathcal{S}} p_{\theta \theta'} \int_{[0,\infty)}e^{-\gamma\left( g_2(f(\theta,y,z),\theta')\right)}\nu(dz) \right] .
\end{eqnarray*}
\end{lem}
\begin{proof} We fix $\theta \in \mathcal{S}$ and set $X :=f(\theta,y,\xi)$,  $h(X,\bar{\theta}):=e^{-\gamma g_1(X,\bar{\theta})}$  and $g(X,\bar{\theta}):=e^{-\gamma g_2(X,\bar{\theta})}$ where $\xi$ and $\bar{\theta}$ are random variables on $[0,\infty)$ and $\mathcal{S}$ with distributions $\nu$ and $p_{\theta \cdot}$ respectively. From their definitions, the aforementioned $h$ and $g$ functions are non-increasing in $X$. Hence, we can invoke Proposition \ref{cov} in the Appendix using the conditional expectation given $\bar{\theta}$. That gives
\begin{align*}
\sum_{\theta'\in\mathcal{S}} p_{\theta \theta'} \int_{[0,\infty)} e^{-\gamma\left( g_1(f(\theta,y,z),\theta')+ g_2(f(\theta,y,z),\theta')\right)} \nu(dz) &
= \sum_{\theta'\in\mathcal{S}} p_{\theta \theta'}E\left[h(X,\bar{\theta}) g(X,\bar{\theta}) |\bar \theta=\theta'\right] \\
&
\ge E \Big( E\left[h(X,\bar{\theta})|\bar \theta\right] E\left[g(X,\bar{\theta}) |\bar \theta\right]\Big).
\end{align*}
Moreover, $E[h(X,\bar{\theta})|\bar \theta]$ and $E[g(X,\bar{\theta})|\bar \theta]$ are independent as one of $h$ and $g$ is constant in $\bar \theta$. Consequently, by the tower property of conditional expectations, we get
$$  E \Big( E\left[h(X,\bar{\theta})|\bar \theta\right] E\left[g(X,\bar{\theta}) |\bar \theta\right]\Big) =   E \Big( E\left[h(X,\bar{\theta})|\bar \theta\right]\Big) E \Big( E\left[g(X,\bar{\theta}) |\bar \theta\right]\Big)=  E \big( h(X,\bar{\theta})\big) E \big( g(X,\bar{\theta})\big).
$$
Thus by summarising the above two inequalities
\begin{align*}
&\sum_{\theta'\in\mathcal{S}} p_{\theta \theta'} \int_{[0,\infty)} e^{-\gamma\left( g_1(f(\theta,y,z),\theta')+ g_2(f(\theta,y,z),\theta')\right)} \nu(dz) \\
& \geq  \left[ \sum_{\theta'\in\mathcal{S}} p_{\theta \theta'}  \int_{[0,\infty)} e^{-\gamma\left( g_1(f(\theta,y,z),\theta') \right)} \nu(dz)   \right] \left[ \sum_{\theta'\in\mathcal{S}} p_{\theta \theta'}   \int_{[0,\infty)} e^{-\gamma\left(  g_2(f(\theta,y,z),\theta')\right)} \nu(dz) \right].
\end{align*}
Hence, by taking $- \frac{1}{\gamma}\ln$ on both sides we get the desired inequality. \end{proof}

\noindent For any $v\in {\cal B},$ we  define the operator $L$  as follows
\begin{equation}
\label{oper_1} Lv(x,\theta):=\sup_{y\in [0,x]} \left(u(x-y) -\frac\beta\gamma
  \ln\left[\sum_{\theta'\in\mathcal{S}} p_{\theta \theta'} \int_{[0,\infty)}e^{-\gamma v(f(\theta,y,z),\theta')}\nu(dz) \right] \right)
  \end{equation}
for all $x\in [0,\infty)$ and $\theta \in \mathcal{S}$.

\begin{lem} \label{lem-oper}
The operator $L$ maps ${\cal B}$ into itself and  is a contraction.
\end{lem}
\begin{proof}
We first show that $Lv \in \mathcal{B}$ for each  $v\in {\cal B}$. Indeed,   by (U2), we obtain
\begin{align}\label{eqn-LvEstimate}
    \nonumber \|Lv\|_w & = \sup_{x \in [0,\infty), \theta \in \mathcal{S}}  \frac{|Lv(x,\theta)|}{w(x,\theta)} 
    \\
   \nonumber  & \leq \sup_{x \in [0,\infty), \theta \in \mathcal{S}}\left[ \frac{d \sup_{y \in [0,x]} |w((x-y),\theta)| }{w(x,\theta)}  \right] + \sup_{x \in [0,\infty), \theta \in \mathcal{S}}\left[ \frac{\beta \sup_{y \in [0,x]} |\widehat{v}(y,\theta) |}{w(x,\theta)}  \right] \\
    & \leq d\sup_{x \in [0,\infty), \theta \in \mathcal{S}}\left[ \frac{  |w(x,\theta)| }{w(x,\theta)}  \right] + \beta \sup_{x \in [0,\infty), \theta \in \mathcal{S}}\left[ \frac{ \sup_{y \in [0,x]} |\widehat{v}(y,\theta) |}{w(x,\theta)}  \right].
\end{align}
By following the similar steps as in the proof of property (iii) in Theorem \ref{P} we arrive at the following
\begin{align}\label{eqn-LvEstimate1}
\nonumber   | \widehat{v}(y,\theta) | & = -\frac 1\gamma\ln \mathbb{E}\left[ e^{-\gamma v(f(\theta_1,y_1,\xi_1),\theta_2)}\mid y_1=y, \theta_1=\theta] \right] \\
\nonumber   & \le -\frac 1\gamma \mathbb{E}\left[ {-\gamma v(f(\theta_1,y_1,\xi_1),\theta_2)}\mid y_1=y, \theta_1=\theta \right]  \\
\nonumber   & =  \mathbb{E}\left[ {v(f(\theta_1,y_1,\xi_1),\theta_2)}\mid y_1=y, \theta_1=\theta \right] \\
   & = \sum_{\theta'\in\mathcal{S}} p_{\theta \theta'} \int_{[0,\infty)}  v(f(\theta,y,z),\theta') \nu(dz).
\end{align}
Therefore using the above inequality \eqref{eqn-LvEstimate1}
\begin{align*}
     \sup_{x \in [0,\infty), \theta \in \mathcal{S}}\left[ \frac{ \sup_{y \in [0,x]} |\widehat{v}(y,\theta) |}{w(x,\theta)}  \right] & \leq \!\! \sup_{x \in [0,\infty), \theta \in \mathcal{S}} \frac{1}{w(x,\theta)} \sup_{y \in [0,x]} \sum_{\theta'\in\mathcal{S}} p_{\theta \theta'}\! \int_{[0,\infty)}\!\!\!\!\!\! \frac{v(f(\theta,y,z),\theta') w(f(\theta,y,z),\theta') }{w(f(\theta,y,z),\theta')} \nu(dz) \\
    & \leq  \|v\|_w \sup_{x \in [0,\infty), \theta \in \mathcal{S}} \frac{1}{w(x,\theta)}  \sup_{y \in [0,x]}\sum_{\theta'\in\mathcal{S}} p_{\theta \theta'} \int_{[0,\infty)}  w(f(\theta,y,z),\theta')\nu(dz) \\
    & \leq \alpha  \|v\|_w.
\end{align*}
Here in the last step we have used (F2).
Hence by substituting the above into  \eqref{eqn-LvEstimate} and using the definition of $\| \cdot\|_w$-norm we obtain
\begin{align*}
    \|Lv\|_w & \leq d+\alpha\beta\|v\|_w.
\end{align*}

\noindent Moreover,  by (U1) and Lemma \ref{lem-basic}, for each $\theta \in
\mathcal{S}$, we have that  $Lv(\cdot,\theta)$ is non-negative.\\

\noindent By (U1) and Lemma \ref{lem-basic}, the function $(x,y) \mapsto u(x-y)+\beta\widehat{v}(y,\theta)$ is continuous. Moreover, the constrained correspondence set $\Gamma: [0,\infty) \to 2^{[0,\infty)}$ given by $\Gamma(x) = [0,x]$ is compact valued, where $2^{[0,\infty)}$ is the power set of $[0,\infty)$. Therefore, $Lv(\cdot,\theta)$ is continuous as a consequence of the Maximum Theorem 
\cite[pg. 235]{Rang} provided $\Gamma$ is hemicontinuous. To show that the correspondence $\Gamma$ is hemicontinuous, we need to show that it is both upper and lower hemicontinuous.
For showing the upper hemicontinuity let $V \subset [0,\infty)$ be an open neighbourhood of $\Gamma(x') = [0,x'].$ Then there exists $\varepsilon >0$ such that
$$ [0,x'] \subset (-\varepsilon, x'+\varepsilon) \subset V.$$
Select the neighborhood $U:= (x'-\varepsilon, x'+\varepsilon)$ of $x'$. Then for all $x \in U$, $x < x' +\varepsilon$ and consequently
$$ \Gamma(x) = [0,x] \subset (-\varepsilon, x'+\varepsilon) \subset V.$$
Thus the correspondence $\Gamma$ is upper hemicontinuous. To prove the lower hemicontinuity let $x' \in [0,\infty)$ and $V \subset [0,\infty)$ be an open set such that $[0,x'] \cap V \neq \emptyset.$ This implies that there exists $0 < x'' < x'$ such that $x'' \in [0,x'] \cap V$. By selecting $U:= (x'',x'+1)$, observe that $x' \in U$, and also for all $x \in U$, $x'' < x$. Therefore,
$$x''\in [0,x] = \Gamma(x) \quad \forall x\in U.$$
As $x'' \in V$, $\Gamma(x) \cap V \neq \emptyset$ for all $x \in U$.
Hence $\Gamma$ is lower hemicontinuous. Consequently, $\Gamma$ is hemicontinuous. Next,  since $u$ is increasing by (U1), we observe that for $x'<x'',$
we have
\begin{align*}
    Lv(x',\theta) &=  \sup_{y\in[0,x']}\left(u(x'-y)+\beta\widehat{v}(y,\theta) \right) \le \sup_{y\in[0,x']}\left(u(x''-y)+\beta\widehat{v}(y,\theta) \right) \\
    & \le \sup_{y\in[0,x'']}\left(u(x''-y)+\beta\widehat{v}(y,\theta)\right) =  Lv(x'',\theta),
\end{align*}
for each $\theta \in {\cal S}$. So $Lv(\cdot,\theta)$ is non-decreasing. 
We now show the concavity of $Lv(\cdot,\theta)$. 
From \eqref{oper_1} we get
\begin{equation} \label{oper_L}
Lv(x,\theta) \ge   u(x-y) -\frac\beta\gamma
  \ln\left[\sum_{\theta'\in\mathcal{S}} p_{\theta \theta'} \int_{[0,\infty)}e^{-\gamma v(f(\theta,y,z),\theta')}\nu(dz) \right],
\end{equation}
for any $y\in [0,x]$. Let us fix $x:=\lambda x'+(1-\lambda) x''$, where $\lambda\in(0,1),$ and $x',x''\in[0,\infty)$. Also let $y'\in [0,x']$ and  $y''\in [0,x'']$ denote
the maximizer in \eqref{oper_1}, which  exists due to the continuity of $y \mapsto u(x-y)+\beta\widehat{v}(y,\theta)$. Again, $\lambda y'+(1-\lambda) y'' (=y$ say) is in $[0,x]$.
Hence, by (U1) we obtain
\begin{equation} \label{oper_u}
u(x-y)=u(\lambda (x'-y')+(1-\lambda)(x''-y''))> \lambda u(x'-y')+(1-\lambda)u(x''-y'').
\end{equation}
Now combining (\ref{oper_L}) with (\ref{oper_u}) and (\ref{oper_rho}) and utilizing the fact that $y',y''$ are maximizers, in $Lv(x',\theta)$ and $Lv(x'',\theta)$ respectively, we finally obtain
\begin{align}\label{Lvx}
\nonumber Lv(x)& > \lambda u(x'-y')+(1-\lambda) u(x''-y'') -\frac\beta\gamma
  \ln\left[\sum_{\theta'\in\mathcal{S}} p_{\theta \theta'} \int_{[0,\infty)}e^{-\gamma v(f(\theta,y,z),\theta')}\nu(dz) \right] \\
\nonumber  & \geq \lambda u(x'-y')+(1-\lambda)u(x''-y'') + \beta \lambda \widehat{v}(y',\theta) +  \beta (1-\lambda) \widehat{v}(y'',\theta) \\
  & = \lambda Lv(x',\theta)+(1-\lambda) Lv(x'',\theta).
\end{align}
Hence we have proved that $L$ maps $\cal{B}$ into itself. Now we prove that $L$ is a contraction on $\cal{B}$. Assume that $v_1,v_2\in {\cal B}.$ Then, for every $x \in [0,\infty)$ and $\theta \in \mathcal{S}$,
\begin{align*}
& Lv_1(x,\theta)-Lv_2(x,\theta) \\
& \le \sup_{y\in [0,x]} \left( -\frac\beta\gamma
  \ln\left[\sum_{\theta'\in\mathcal{S}} p_{\theta \theta'} \int_{[0,\infty)}e^{-\gamma v_1(f(\theta,y,z),\theta')}\nu(dz) \right] + \frac\beta\gamma
  \ln\left[\sum_{\theta'\in\mathcal{S}} p_{\theta \theta'} \int_{[0,\infty)}e^{-\gamma v_2(f(\theta,y,z),\theta')}\nu(dz) \right]  \right)\\
& \le \beta \sup_{y\in [0,x]} \left( -\frac 1 \gamma
  \ln\left[\sum_{\theta'\in\mathcal{S}} p_{\theta \theta'} \int_{[0,\infty)}e^{-\gamma \|v_1-v_2\|_w \sup_{\theta'\in\mathcal{S}} \big(w(f(\theta,y,z),\theta')\big) -\gamma v_2(f(\theta,y,z),\theta') }\nu(dz) \right]\right.\\
&\qquad\qquad\qquad + \left.\frac 1 \gamma
  \ln\left[\sum_{\theta'\in\mathcal{S}} p_{\theta \theta'} \int_{[0,\infty)}e^{-\gamma v_2(f(\theta,y,z),\theta')}\nu(dz) \right]  \right)
\end{align*}
which follows from property (ii) in Theorem \ref{P}. Using the fact that $w$ and $v_2$ are non-decreasing in first variable and $\sup_{\theta'\in\mathcal{S}} \big(w(f(\theta,y,z),\theta')\big)$ is constant in $\theta'$  we apply  Lemma \ref{lem-covl} to assert that the above is less than or equal to
\begin{align*}
& \beta \sup_{y\in [0,x]} \left( -\frac 1 \gamma
  \ln\left[\sum_{\theta'\in\mathcal{S}} p_{\theta \theta'} \int_{[0,\infty)}e^{-\gamma \|v_1-v_2\|_w \sup_{\theta'\in\mathcal{S}} \big(w(f(\theta,y,z),\theta')\big)  }\nu(dz) \right]  \right. \\
& \qquad\qquad\qquad \left.-\frac 1 \gamma
  \ln\left[\sum_{\theta'\in\mathcal{S}} p_{\theta \theta'} \int_{[0,\infty)}e^{ -\gamma v_2(f(\theta,y,z),\theta') }\nu(dz) \right]  + \frac 1 \gamma
  \ln\left[\sum_{\theta'\in\mathcal{S}} p_{\theta \theta'} \int_{[0,\infty)}e^{-\gamma v_2(f(\theta,y,z),\theta')}\nu(dz) \right]  \right)\\
&= \beta \sup_{y\in [0,x]} \left( -\frac 1 \gamma
  \ln\left[\sum_{\theta'\in\mathcal{S}} p_{\theta \theta'} \int_{[0,\infty)}e^{-\gamma \|v_1-v_2\|_w \sup_{\theta'\in\mathcal{S}} \big(w(f(\theta,y,z),\theta')\big)  }\nu(dz) \right]   \right) \\
&\le   \beta \sup_{y\in [0,x]} \left(
   \sum_{\theta'\in\mathcal{S}} p_{\theta \theta'} \int_{[0,\infty)}  \|v_1-v_2\|_w \sup_{\theta'\in\mathcal{S}} \big(w(f(\theta,y,z),\theta')\big)  \nu(dz)    \right)
\end{align*}
by (iii) in Theorem \ref{P}. By (F2) above is not greater than
$\alpha \beta \|v_1-v_2\|_w w(x,\theta)$.
By changing the roles of $v_1$ with $v_2$ we finally obtain
$$\|Lv_1-Lv_2\|_w\le \alpha\beta\|v_1-v_2\|_w.$$
Again, from (F2), $\alpha \in (0, 1/\beta)$, and so $\alpha\beta < 1$. Consequently, $L$ is a contraction. This completes the proof.
\end{proof}

\begin{proof}[Proof of Theorem \ref{mainthm}]
{\bf Part (a).} Clearly, a solution of (\ref{oe}), if exists, is a fixed point of $L$ (as in \eqref{oper_1}) and vice versa. The latter exists uniquely from Lemma \ref{lem-oper} and the Banach fixed point theorem
applied to the operator $L$.
Hence, \eqref{oe} has a unique solution $V\in {\cal B}$. In addition, note that \eqref{Lvx} has strict inequality making $LV$ strictly concave. Since $V$ and $LV$ are identical, $V$ is strictly concave too.\\

\noindent {\bf Part (b).} Due to the strict concavity of $V$ (Part (a)), $\widehat{V}$, which is given by \eqref{b} is also strictly concave, continuous and non-decreasing by invoking an application of Lemma \ref{lem-basic}. Due to the strictly concavity, for each $x\in[0,\infty)$ and $\theta \in \mathcal{S}$,
the set of maximizers on the right-hand side  of (\ref{oe}) is a singleton. Hence, the multifunction becomes a unique function $\phi^*\in\Phi$ (say) that maximizes the right side of  \eqref{oe}. Consequently \eqref{oe_max} is true.\\

\noindent We now show $\phi^*$ to be continuous. To this end we first recall the applicability of Maximum Theorem which has been asserted in the proof of Lemma \ref{lem-oper}. As a further consequence of the Maximum Theorem the map $x\mapsto y^*(x)$, which gives the set of maximizers of $y \mapsto u(x-y)+\beta\widehat{v}(y,\theta)$, is compact valued and upper hemicontinuous.
As a singleton valued upper hemicontinuous function is continuous, ${\phi}^*\in\Phi$ is also a continuous function. Therefore, the optimal consumption  strategy $c^*(x,\theta)=x-{\phi}^*(x,\theta)$ is also continuous in $x$.\\

\noindent We fix two points $x'< x''\in [0,\infty).$ Suppose there exists a $\theta  \in \mathcal{S}$ such that ${\phi}^*(x',\theta)>{\phi}^*(x'',\theta)$. Then, using $u(x'')-u(x''-\delta)< u(x')-u(x'-\delta)$, which is true for every $0<\delta<x'$ due to the strict concavity of $u$, we get
\begin{align*}
u(x''-{\phi}^*(x'',\theta))-u(x''-{\phi}^*(x',\theta))
< u(x'-{\phi}^*(x'',\theta))-u(x'-{\phi}^*(x',\theta)).
\end{align*}
By adding $\beta \widehat{V}({\phi}^*(x'',\theta),\theta) -\beta\widehat{V}({\phi}^*(x',\theta),\theta)$ on both sides and rearranging
\begin{align*}
&\Big(u(x''-{\phi}^*(x'',\theta)) +\beta \widehat{V}({\phi}^*(x'',\theta),\theta)\Big)- \Big(u(x''-{\phi}^*(x',\theta)) +\beta \widehat{V}({\phi}^*(x',\theta),\theta)\Big)\\
&< \Big( u(x'-{\phi}^*(x'',\theta))+\beta \widehat{V}({\phi}^*(x'',\theta),\theta)\Big) -\Big(u(x'-{\phi}^*(x',\theta)) +\beta\widehat{V}({\phi}^*(x',\theta),\theta)\Big).
\end{align*}
Using \eqref{oe} and \eqref{oe_max}, clearly the left side is greater or equal to $ LV(x'',\theta)-LV(x'',\theta) =0$ and the right side is less than or equal to $ LV(x',\theta) - LV(x',\theta) =0$. This leads to a contradiction. Hence we have proved by contradiction that  ${\phi}^*(x',\theta)\le {\phi}^*(x'',\theta)$ for all $x'< x''\in [0,\infty)$ and every $\theta  \in \mathcal{S}$.\\

\noindent Furthermore, observe that (\ref{oe}) can be re-written as follows
$$V(x,\theta)=\sup_{a\in [0,x]}\left(u(a)-
\beta\widehat{V}(x-a,\theta)\right),\quad x\in[0,\infty).$$
Recall that $c^* (x, \theta) := x - \phi^* (x, \theta)$. Using (U1), the strict concavity of $\widehat{V}$ and arguments similar to the above, one gets that the optimal consumption  strategy $c^* (x, \theta)\in\Phi$ is also non-decreasing in $x$ for all $\theta \in {\cal S}$.\\

\noindent {\bf Part (c).} From (\ref{oe}) it follows that
 \begin{align}\label{ineqV}
     V(x,\theta)\ge u(x-y)-
  \frac\beta\gamma\ln\left[\sum_{\theta'\in\mathcal{S}} p_{\theta \theta'} \int_{[0,\infty)} e^{-\gamma V(f(\theta,y,z),\theta')}\nu(dz) \right],
 \end{align}
for all $ y\in[0,x]$, $x\in[0,\infty)$ and $\theta \in \mathcal{S}$. Let us set, for each $k \in \mathbb{N}$, $\bar{V}_{k}: B_w(H_{k})\to \mathbb{R}$ given by $$\bar{V}_{k}(h_{k}):=  V(x_{k},\theta_{k}).$$
Note that $\bar{V}_1 = V$. In view of this, by letting $\pi=(\pi_k)_{k\in\mathbb{N}}\in\Pi$ be any investment policy, for any history $h_k\in H_k,$ $k\in\mathbb{N},$
\eqref{ineqV} implies that
\begin{align}\label{L_iter}
    (L_{\pi_k}\bar{V}_{k+1})(h_k)&  = u(x_k - \pi_k(h_k)) -\frac \beta \gamma \ln \left[\sum_{\theta'\in\mathcal{S}} p_{\theta_k \theta'} \int_{[0,\infty)} e^{-\gamma \bar{V}_{k+1}(h_k,\pi_k(h_k), f(\theta_k,\pi_k(h_k),z),\theta')}\nu(dz) \right] \nonumber\\
    & = u(x_k - \pi_k(h_k)) -\frac \beta \gamma \ln \left[\sum_{\theta'\in\mathcal{S}} p_{\theta_k \theta'} \int_{[0,\infty)} e^{-\gamma V(f(\theta_k,\pi_k(h_k),z),\theta')}\nu(dz) \right] \nonumber\\
    & \leq V(x_k,\theta_k) = \bar{V}_k(h_k).
\end{align}
Consequently, since $L_{\pi_k}$ is monotone we get,
\begin{equation}\label{L_iter1}
V(x,\theta) =\bar{V}_{1}(h_1)\ge (L_{\pi_{1}}\bar{V}_{2})(h_{1}) \ge  [L_{\pi_{1}} ( L_{\pi_2}\bar{V}_{3})] (h_{1}) \ge\cdots \ge [(L_{\pi_{1}} \circ\cdots \circ L_{\pi_T})\bar{V}_{T+1}] (h_{1}),
\end{equation}
for $h_1 = (x,\theta)$.
Since $V\ge 0$ and  $L_{\pi_k}$ is monotone for every $\pi_k,$ $k\in\mathbb{N},$ we obtain
$$
V(x,\theta)\ge [(L_{\pi_1}\circ \cdots \circ L_{\pi_T} )\bar{V}_{T+1}](x,\theta)\ge  [(L_{\pi_1}\circ \cdots \circ L_{\pi_T} ){\bf 0}](x,\theta) =J_{T}(x,\theta,\pi),
$$
for any $\pi\in\Pi$, $x\in[0,\infty)$ and $\theta \in \mathcal{S}$. Here the last equality follows from \eqref{J_T}. Since the above is true for every $T>0$, by \eqref{J}, we get $V(x,\theta)\ge \lim_{T\to \infty} J_{T}(x,\theta,\pi) =J(x,\theta,\pi)$ for any $\pi\in\Pi$, $x\in[0,\infty)$ and $\theta \in \mathcal{S}$.
Hence,
\begin{equation}
\label{it}
V(x,\theta)\ge \sup_{\pi\in\Pi}J(x,\theta,\pi),\quad x\in [0,\infty), \theta \in \mathcal{S}.
\end{equation}

\noindent We next aim to establish the reverse inequality. Let ${\phi}^*: [0,\infty) \times \mathcal{S} \to [0,\infty)$ be as in (\ref{oe_max}). We set ${i}^* = \{{i}^*_k\}_{k \in \mathbb{N}}$ such that $i_k^*(h_k) = {\phi}^* (x_k,\theta_k)$ for all $k \in \mathbb{N}$.
For convenience of notation we set
$u_{{\phi}^*}(x,\theta):= u(x - {\phi}^*(x,\theta))$ and for any non-negative function $\varphi \in B_w$, we define:
\begin{align*}
\rho_{{\phi}^*} (\varphi)(x,\theta):=& - \frac1\gamma\ln\sum_{\theta'\in\mathcal{S}} p_{\theta \theta'} \int_{[0,\infty)}e^{-\gamma \varphi(f(\theta,\phi^*(x,\theta),z),\theta')}\nu(dz),\\
L_{{\phi}^*}(\varphi)(x,\theta):=& u_{{\phi}^*}(x,\theta) + \beta \rho_{{\phi}^*}(\varphi)(x,\theta).
\end{align*}
Then
 the right-hand side of \eqref{oe_max} can be re-written as
$$ u_{{\phi}^*}(x,\theta)+\beta \rho_{{\phi}^*} (V)(x,\theta) =L_{{\phi}^*}(V)(x,\theta),$$
for $(x,\theta) \in [0,\infty) \times \mathcal{S}$. Consequently, the equality \eqref{oe_max} can be rewritten as follows:
\begin{equation}\label{eqn-V2}
    V(x,\theta) = L_{{\phi}^*}(V)(x,\theta).
\end{equation}

\noindent Hence, by iterating the last equality $T-1$ times we get that
\begin{equation}\label{L_iter_i}
 V(x,\theta)= L_{{\phi}^*}^{(T)}(V)(x,\theta), \qquad x\in [0,\infty), \theta \in \mathcal{S},
 \end{equation}
where $L_{{\phi}^*}^{(T)}$ denotes the $T$-th composition of the operator $L_{{\phi}^*}$ with itself. Thus, (\ref{L_iter_i}) with property (ii) in Theorem \ref{P} followed by \eqref{eqn-LvEstimate1} together with the Assumption (F2) yield that
\begin{eqnarray}
\label{it2}
V(x,\theta)&=&  L_{{\phi}^*}^{(T-1)} (u_{{\phi}^*}
+\beta\rho_{{\phi}^*}(V)) (x,\theta) \le
  L_{{\phi}^*}^{(T-1)}
(u_{{\phi}^*}+\beta \rho_{{\phi}^*}(\|V\|_w \bar{w})) (x,\theta) \nonumber \\
&\le&
 L_{{\phi}^*}^{(T-1)}  (u_{{\phi}^*}+\alpha \beta\|V\|_w \bar{w}) (x,\theta)\nonumber \\
&=&
 L_{{\phi}^*}^{(T-2)} (u_{{\phi}^*}+\beta \rho_{{\phi}^*} (u_{{\phi}^*}+\alpha\beta\|V\|_w \bar{w})) (x,\theta),
\end{eqnarray}
for $x\in [0,\infty),\theta \in \mathcal{S}$ and $\bar{w}(\cdot) := \max_{\theta \in \mathcal{S}}w(\cdot,\theta)$.
Now by putting $g_1:=u_{{\phi}^*}=J_1(\cdot,\cdot, {\phi}^*),$ $g_2:=\alpha\beta\|V\|_w \bar{w}$ in Lemma
\ref{lem-covl}, we have  that
\begin{eqnarray}\label{it1}
  \beta\rho_{{\phi}^*}(u_{{\phi}^*}
+\alpha\beta\|V\|_w \bar{w})(x,\theta)&\le& \beta\rho_{{\phi}^*}(u_{{\phi}^*})(x,\theta)+ \beta\rho_{{\phi}^*}(\alpha\beta\|V\|_w \bar{w})(x,\theta) \nonumber \\
&\le& \beta\rho_{{\phi}^*}(u_{{\phi}^*})(x,\theta)+(\alpha\beta)^2\|V\|_w \bar{w}(x), \quad x\in[0,\infty), \theta \in \mathcal{S},
\end{eqnarray}
where the second inequality is due to \eqref{eqn-LvEstimate1} and the Assumption (F2).
Next combining the inequalities (26) and (27) gives:
\begin{eqnarray}
\label{it3}
V(x,\theta) &\le&  L_{{\phi}^*}^{(T-2)} (u_{{\phi}^*}+ \beta\rho_{{\phi}^*}(u_{{\phi}^*})+(\alpha\beta)^2\|V\|_w \bar{w}) (x,\theta) \nonumber\\ &=&
 L_{{\phi}^*}^{(T-3)} (u_{{\phi}^*}+ \beta\rho_{{\phi}^*} (u_{{\phi}^*}+ \beta\rho_{{\phi}^*} (u_{{\phi}^*})+(\alpha\beta)^2\|V\|_w \bar{w})) (x,\theta).
\end{eqnarray}
Now we repeat the above procedure as follows: from Lemma \ref{lem-incr} (see also \eqref{rem-J2})  the function
$$x\mapsto J_2(x,\theta,{\phi}^*)=u_{{\phi}^*}(x,\theta)+\beta\rho_{{\phi}^*}(u_{{\phi}^*})(x,\theta)$$
is non-decreasing. Hence, making use  again of
Lemma \ref{lem-covl} ($g_1:=J_2(\cdot,\cdot,{\phi}^*),$ $g_2:=(\alpha\beta)^2\|V\|_w \bar{w}$), \eqref{eqn-LvEstimate1} and (F2) we get
\begin{align} \label{it4}
& \beta\rho_{{\phi}^*} (u_{{\phi}^*}+ \beta\rho_{{\phi}^*}(u_{{\phi}^*})+(\alpha\beta)^2\|V\|_w \bar{w})(x,\theta)  = \beta\rho_{{\phi}^*}(J_2(\cdot,\cdot,{\phi}^*)+(\alpha\beta)^2\|V\|_w \bar{w})(x,\theta) \nonumber \\
 & \qquad  \le \beta\rho_{{\phi}^*}(J_2(\cdot,\cdot,{\phi}^*))(x,\theta)+\beta\rho_{{\phi}^*}((\alpha\beta)^2\|V\|_w \bar{w}) (x,\theta)\nonumber \\
& \qquad  \le \beta\rho_{{\phi}^*}(J_2(\cdot,\cdot,{\phi}^*))(x,\theta)+(\alpha\beta)^3\|V\|_w \bar{w}(x),\qquad x\in[0,\infty), \theta \in \mathcal{S}.
\end{align}
By combining (\ref{it3}) and (\ref{it4}) we obtain that
\begin{eqnarray*}
\label{it5}\nonumber
V(x,\theta)&\le&   L_{{\phi}^*}^{(T-3)}
 \left(u_{{\phi}^*}+ \beta\rho_{{\phi}^*}(J_2(\cdot,\cdot,{\phi}^*))+(\alpha\beta)^3\|V\|_w \bar{w}\right) (x,\theta)\\
 &=&  L_{{\phi}^*}^{(T-3)}\left(J_3(\cdot,\cdot,{\phi}^*)+(\alpha\beta)^3\|V\|_w \bar{w}\right) (x,\theta)\\
 &=&  L_{{\phi}^*}^{(T-4)}  \left(u_{{\phi}^*}+ \beta\rho_{{\phi}^*}(J_3(\cdot,\cdot,{\phi}^*)+(\alpha\beta)^3\|V\|_w \bar{w})\right) (x,\theta).
\end{eqnarray*}
Repeating this procedure, i.e., making use of Lemma \ref{lem-covl}  for the functions $g_1=J_k(\cdot,\cdot,{\phi}^*)$
(by Lemma \ref{lem-incr} it  is non-decreasing)  and $g_2:=(\alpha\beta)^k\|V\|_w \bar{w}$ for $k=3, \ldots, T-1$
 we finally deduce
\begin{equation}
\label{ii}
V(x,\theta) \le J_{T}(x,\theta,{\phi}^*)+(\alpha\beta)^{T}\|V\|_w \bar{w}(x),\quad x\in[0,\infty), \theta \in \mathcal{S}.
\end{equation}
Thus, letting $T\to\infty$ in \eqref{ii} and by noting that
$\alpha \beta < 1,$ it follows that
\begin{equation}
\label{itit}
V(x,\theta)\le J(x,\theta, {\phi}^*)\quad x\in [0,\infty), \theta \in \mathcal{S}.
\end{equation}
Now, (\ref{it}) and (\ref{itit}) combined together yield part (c). \hfill $\Box$
\end{proof}

\section{Euler Equation for Consumption}
\noindent The relation between the gain of utility of present consumption and the discounted expected value of that of the future consumption is expressed in Euler equation. The equation generally states that the gain in utility for saving now and then consuming in future, instead of immediate consumption should not, after discounting, be less or more than the utility gain of consuming now. This section is devoted to establish the Euler equation for optimal consumption. We impose additional conditions
that guarantee differentiability of functions which are required in formulating the model.

\begin{enumerate}
  \item[(U3)] The function $u:[0,\infty)\to [0,\infty)$ is continuously differentiable
  on $(0,\infty).$
  \item[(U4)] $u'_+(0)=\infty,$ where $u^{\prime}_+ (0)$ is the right-hand derivative of $u (x)$ at $x = 0$.
  \item[(F3)] The  function $f(\theta,\cdot,z):[0,\infty) \to [0,\infty)$ is
   continuously differentiable on $(0,\infty)$, for each $z\in [0,\infty), \theta\in \mathcal{S}$.
    \item[(F4)] $f(\theta,0,z)=0$ for all $z\ge 0$ and $\theta \in \mathcal{S}$.
    \item[(F5)] For almost every $z$ w.r.t. $\nu$, there is an investment  $y>0$ such that  $f'(\theta,y,z)>0,$
 for all $\theta\in \mathcal{S}$ where $f'(\theta,y,z):=\frac{\partial f(\theta,y,z)}{\partial y}.$
 \end{enumerate}

\noindent Assumption (F5) together with concavity (F1) imply that $f'_+(\theta,0,z)>0$ $\nu$-a.s. for every $\theta\in \mathcal{S}$. This also implies, in view of (F4), $f(\theta,y,z)>0$ for all $y>0$ for almost every $z$ w.r.t. $\nu$.

\begin{thm}\label{thm2} Assume (U1)-(U4) and (F1)-(F5). Then, we have the following.
\begin{itemize}
 \item[(a)]  For any $x\in (0,\infty)$ and $\theta \in \mathcal{S}$ the following Euler equation holds
 \begin{align}
  \label{e}
\nonumber u'(c^*(x,\theta))=\beta \frac{\sum_{\theta'\in \mathcal{S}} p_{\theta \theta'}  \int_{[0,\infty)} u'(c^*(f(\theta, {\phi}^*(x,\theta),z), \theta'))f'(\theta,{\phi}^*(x,\theta),z)e^{-\gamma
 V(f(\theta,{\phi}^*(x,\theta),z),\theta')} \nu(dz)}
  {\sum_{\theta'\in \mathcal{S}} p_{\theta \theta'} \int_{[0,\infty)}e^{-\gamma V(f(\theta,{\phi}^*(x,\theta),z),\theta')}\nu(dz)},\\
  \end{align}
  where $V$ is the value function obtained in Theorem \ref{mainthm}.
  \item[(b)] For each $\theta \in \mathcal{S}$, the functions $x\mapsto {\phi}^*(x,\theta)$ and $x\mapsto c^*(x,\theta)$ are increasing.
\end{itemize}
\end{thm}

\noindent It is important to note that the above Euler equation for consumption involves the value function $V$. This is not the case when the objective is to maximize the standard expected utility. It is evident that setting $\gamma= 0$ in \eqref{e} removes the $V$ dependence, and coincides with the well-known Euler equation for the model with the expected utility. We also note that for a nonzero $\gamma$, the right side involves an expectation w.r.t. a different probability measure that depends on the value function. In the subsequent lemmas we shall assume that the conditions (U1)-(U4) and (F1)-(F5) hold true.

\begin{lem} \label{lem_b}
Let $\widehat{V}$ be as in \eqref{b} then
\begin{equation} \label{lem_bb}
\widehat{V}'_+(0,\theta)=\infty \quad \textrm{for each } \quad \theta \in {\cal S},
\end{equation}
where $\widehat{V}'(y,\theta):=\frac{\partial \widehat{V}(y,\theta)}{\partial y}$.
\end{lem}

\begin{proof} Let us fix $\theta \in \mathcal{S}$.
To prove \eqref{lem_bb}, take any sequence  $y_n\to 0^+$ as $n\to\infty.$
By Assumptions (F4) and (U1) and the fact that ${\phi}^*(0,\theta) =0$, we get from \eqref{oe_max} that $V(0,\theta)=0$ and hence $\widehat{V}(0,\theta)=0$ from \eqref{b}.
Again, from \eqref{ineqV}, using non-negativity of $V$, and Theorem \ref{P}(i), we get $V(x,\theta) \ge u(x)$ by setting $y=0$.
Consequently, by (ii) and (vi) of Theorem \ref{P} and \eqref{b} we have
\begin{align}\label{lem_b0}
\frac{\widehat{V}(y_n,\theta)-\widehat{V}(0,\theta)}{y_n}& = -\frac1{\gamma y_n}
  \ln\sum_{\theta'\in \mathcal{S}} p_{\theta \theta'} \int_{[0,\infty)}e^{-\gamma V(f(\theta,y_n,z),\theta')}\nu(dz) \nonumber\\
  & \ge -\frac1{\gamma y_n}
  \ln\sum_{\theta'\in \mathcal{S}} p_{\theta \theta'} \int_{[0,\infty)}e^{-\gamma u(f(\theta,y_n,z))}\nu(dz) \nonumber\\
  &  \ge -\frac1{\gamma}
  \ln\sum_{\theta'\in \mathcal{S}} p_{\theta \theta'} \int_{[0,\infty)}e^{-\gamma \frac{u(f(\theta,y_n,z))}{y_n}}\nu(dz)
\end{align}
for all $y_n\in (0,1]$. Recall that from (U1), $u(0) = 0$ and that from (F4), $f (\theta, 0, z) = 0$. Then for any $z \ge 0$,
\begin{equation}
\nonumber\frac{ u(f(\theta,y_n,z))}{y_n}=\frac{ u(f(\theta,y_n,z))-u(f(\theta,0,z))}{f(\theta,y_n,z)-f(\theta,0,z)}\frac{f(\theta,y_n,z)-f(\theta,0,z)}{y_n}.
\end{equation}
Letting $n\to\infty$ on both sides of above equality, we obtain that
\begin{equation}\label{lem_b2}
\frac{ u(f(\theta,y_n,z))}{y_n} \to u'_+(f(\theta,0,z))f'_+(\theta,0,z) =  u'_+(0)f'_+(\theta,0,z).
\end{equation}
Since $u(\cdot)$ and $f(\theta,\cdot,z)$
are concave and $u$ is increasing, the composition $u(f(\theta,\cdot,z))$ is concave and thus Lemma \ref{lem-concave1} is applicable to $u(f(\theta,\cdot,z))$ as $u(f(\theta,0,z))=0$. As a result, the sequence in \eqref{lem_b2} is monotone increasing. As the integrands in \eqref{lem_b0} are uniformly bounded by 1, and $\nu$ is a finite measure, the integrals converge using the convergence of \eqref{lem_b2}. To be more precise
\begin{align}\label{inf}
\liminf_{y_n \downarrow 0} \frac{\widehat{V}(y_n,\theta)-\widehat{V}(0,\theta)}{y_n} \ge -\frac1{\gamma}
  \ln\int_{[0,\infty)}e^{-\gamma u'_+(0)f'_+(\theta,0,z)}\nu(dz).
  \end{align}
It is evident from Assumption (F5) that there is a $y'>0$ such that $\nu\{z: f'(\theta,y',z)>0\}=1$ for all $\theta \in \mathcal{S}$. Now due to the concavity of $f$ in $y$, i.e., $f'(\theta,y,z)\ge f'(\theta,y',z)$ for all $y\in (0,y')$ and $\theta \in \mathcal{S}$, we get $\nu\{z: f'(\theta,y,z)>0\}\ge \nu\{z: f'(\theta,y',z)>0\}=1$ for all $y\in (0,y')$ and $\theta \in \mathcal{S}$. Thus $\nu\{z: f'_+(\theta,0,z)>0\}\ge \nu\{z: f'(\theta,y',z)>0\}=1$. Hence by (U4) $e^{-\gamma u'_+(0)f'_+(\theta,0,z)} =0$ almost everywhere w.r.t. $\nu$ for each $\theta \in \mathcal{S}$. Then the right side of \eqref{inf} is positive infinity. Therefore $\widehat{V}'_+(0,\theta)$ exists as infinity.
\end{proof}

\noindent We recall that the optimal policy ${\phi}^*(x,\cdot)$ (as  in \eqref{oe_max}) is in $[0,x]$ for each $x\ge 0$. We show in the following lemma that for each $x>0$, the optimal investment ${\phi}^*(x,\cdot)$ is neither zero nor complete $x$. The proof is motivated from \cite[Lemma A.5]{kami}.

 \begin{lem} \label{lem_kami}
Let ${\phi}^*$ be as in (\ref{oe_max}). Then, ${\phi}^*(x,\theta)\in (0,x)$ for any $x\in(0,\infty)$ and $\theta \in \mathcal{S}$.
\end{lem}

\begin{proof}  We define, for each $x \in [0,\infty)$, and  $\theta \in \mathcal{S}$
\begin{equation}
\nonumber   W(x,\theta,y) := u(x-y) + \beta \widehat{V}(y,\theta), \quad \forall y \in [0,x].
\end{equation}
Then, since ${\phi}^*$ is an optimal policy, we have \begin{equation}\label{kami-2}
    W(x,\theta,y) \le W(x,\theta,{\phi}^*(x,\theta)), \quad \forall y \in [0,x].
\end{equation}
Let if possible ${\phi}^*(x,\theta) = 0$ for a given $x>0$, and $\theta\in {\cal S}$.
Then for $y \in ({\phi}^*(x,\theta),x)$, \eqref{kami-2} gives
\begin{align*}
    \beta \frac{\widehat{V}(y,\theta) - \widehat{V}({\phi}^*(x,\theta),\theta)}{y-{\phi}^*(x,\theta)} \le \frac{u(x-{\phi}^*(x,\theta)) - u(x-y)}{y-{\phi}^*(x,\theta)}.
\end{align*}
Letting $y \downarrow {\phi}^*(x,\theta)$ yields $\beta \widehat{V}'_+({\phi}^*(x,\theta),\theta) \le u'_{-}(x-{\phi}^*(x,\theta))<\infty$. That is  $\widehat{V}'_+(0,\theta)<\infty$, a contradiction to Lemma \ref{lem_b}. Thus ${\phi}^*(x,\theta) \neq 0$ for every $x$ and $\theta$.\\

\noindent Now let if possible ${\phi}^*(x,\theta) = x$. Observe that, for $y \in (0,{\phi}^*(x,\theta))$, \eqref{kami-2}  gives
\begin{align*}
    \beta \frac{\widehat{V}({\phi}^*(x,\theta),\theta) - \widehat{V}(y,\theta)}{{\phi}^*(x,\theta)-y} \ge \frac{u(x-y) - u(x-{\phi}^*(x,\theta))}{{\phi}^*(x,\theta)-y}.
\end{align*}
Letting $y \uparrow {\phi}^*(x,\theta)$ yields $\beta \widehat{V}'_{-}({\phi}^*(x,\theta),\theta) \ge u_{+}'(x-{\phi}^*(x,\theta))$, where $\widehat{V}'_{-}$ is the left hand derivative of $\widehat{V}$, which exists everywhere due to strict concavity of $\widehat{V}$ from Theorem \ref{mainthm}(b). Utilizing ${\phi}^*(x,\theta) =x$,  $u_{+}'(x-{\phi}^*(x,\theta))=\infty$ due to (U4).
This gives $\widehat{V}'_{-}(x,\theta) = \infty$, and this is a contradiction to the finiteness of $\widehat{V}'_{-}(y,\theta), \forall y >0,$ which is implied by the strict concavity of $\widehat{V}$ from Theorem  \ref{mainthm}(b). Hence ${\phi}^*(x,\theta) \neq x$ for every $x$ and $\theta$.
\end{proof}
The proof of following Lemma is along the similar lines as the proofs of
\cite[Proposition 12.1.18 and Corollary 12.1.19]{st}.
\begin{lem} \label{lem_env}
For each $\theta \in \mathcal{S}$, the function $V(\cdot,\theta)$  is continuously differentiable on $(0,\infty)$ and
$V'(x,\theta)=u'(c^*(x,\theta)),$ for $x\in(0,\infty)$. Here $V'(x,\theta) := \frac{\partial V(x,\theta)}{\partial x}$.
\end{lem}

\begin{proof} We first note that $V_+'$ and $V_-'$ exist by strict concavity of $V$ from part (a) of Theorem \ref{mainthm}. Moreover, since $V$ is concave, $V_+'(x,\theta) \le V_{-}'(x,\theta)$  for each  $x>0$, $\theta \in \mathcal{S}$. Thus for proving differentiability, it is enough to show $V_+'\ge  V_{-}'$.

\noindent To this end we fix $x>0$, $\theta \in \mathcal{S}$.
From Lemma \ref{lem_kami} we have ${\phi}^*(x,\theta) < x$. Consequently, there exists an open neighborhood $G$ of zero with
$$ 0 \le {\phi}^*(x,\theta) \le  x + h,  \textrm{ for all } h \in G. $$
Therefore, $W(x+h,\theta, {\phi}^*(x,\theta))$ is defined where for all $y \in [0,x]$, $W$ is given by
\begin{equation}
\nonumber    W(x,\theta,y) = u(x-y) + \beta \widehat{V}(y,\theta) \quad \le V(x,\theta)
\end{equation}
as in the proof of Lemma \ref{lem_kami}. Since ${\phi}^*$ is optimal in Theorem \ref{mainthm} (c), we have $
    W(x+h,\theta, {\phi}^*(x,\theta))  \le  V(x+h,\theta) $ for all $h \in G$. Then, it follows that,
\begin{align}\label{lem_env-2}
     V(x+h,\theta)  - V(x,\theta) & \ge  W(x+h,\theta, {\phi}^*(x,\theta)) - W(x,\theta, {\phi}^*(x,\theta)) \nonumber\\
     & = u(x+h - {\phi}^*(x,\theta)) -  u(x - {\phi}^*(x,\theta)) \qquad \forall h \in G.
\end{align}
By replacing $h$ by $h_n$ in  \eqref{lem_env-2}, where  $\{h_n\}_{n \in \mathbb{N}} \subset G $ is a sequence such that $h_n > 0$ and $ h_n \downarrow 0$,  we obtain
\begin{align*}
     \frac{V(x+h_n,\theta)  - V(x,\theta) }{h_n}& \ge \frac{u(x+h_n - {\phi}^*(x,\theta)) -  u(x - {\phi}^*(x,\theta))}{h_n}, \qquad \forall n \in \mathbb{N}.
\end{align*}
Hence, by letting $n \to \infty$ on the both sides of above inequality, and using the continuous differentiablility of $u(x)$, and continuity of $V(x,\theta)$ in $x$, we get
$V_+'(x,\theta) \ge u'(x-{\phi}^*(x,\theta))$. 
Next, by taking any sequence $\{h_n\}_{n \in \mathbb{N}} \subset G $ such that
$h_n <0 $ and $ h_n \uparrow 0$, and using \eqref{lem_env-2},  we obtain
\begin{align*}
     \frac{V(x+h_n,\theta)  - V(x,\theta) }{h_n}& \le \frac{u(x+h_n - {\phi}^*(x,\theta)) -  u(x - {\phi}^*(x,\theta))}{h_n}, \qquad \forall n \in \mathbb{N}.
\end{align*}
By taking limit $n \to \infty$ on both sides of above inequality we get, as earlier, $V_{-}'(x,\theta) \le u'(x-{\phi}^*(x,\theta))$.
Thus, $V_{-}'(x,\theta) \le u'(x-{\phi}^*(x,\theta)) \le V_+'(x,\theta)$ as desired. Hence, the left and right derivatives are equal (implying differentiability of $V$
at $x$) to $u'(x-{\phi}^*(x,\theta))$. Recall from Theorem \ref{mainthm} that $x-{\phi}^*(x,\theta)$ gives the  optimal consumption $c^*(x,\theta)$.  This completes the proof.
\end{proof}

\begin{proof}[Proof of Theorem \ref{thm2}]
Note that by Theorem \ref{mainthm} (b), ${\hat V}$ defined in Eq. \eqref{b} is strictly concave. Consequently, as in the proof of Lemma \ref{lem_env}, the right-derivative ${\hat V}^{\prime}_+$ and the left-derivative ${\hat V}^{\prime}_-$ exist.
Hence we are required to show their equality and the continuity of the derivative. To this end we set a new measure on the product space $[0, \infty) \times {\cal S}$ by
$$\nu'(\theta,y)(dz,\theta'):= \frac {p_{\theta \theta'} e^{-\gamma V(f(\theta,y,z),\theta')}\nu(dz)}{\sum_{\theta'\in \mathcal{S}} p_{\theta \theta'} \int_{[0,\infty)}e^{-\gamma V(f(\theta,y,z),\theta')}\nu(dz)}$$
for each $y\in(0,\infty)$, $\theta \in \mathcal{S}$ and note that
$$\sum_{\theta'\in \mathcal{S}} \int_{[0,\infty)} \nu'(\theta,y)(dz,\theta')=1.
$$
Hence $\nu'(\theta,y)(\cdot,\cdot)$ is a probability measure. Using the above and Theorem \ref{P}(iii) for any $h>0$
\begin{align*}
\frac{\widehat{V}(y+h,\theta)-\widehat{V}(y,\theta)}{h} &=
-\frac 1{h \gamma }\ln \frac { \sum_{\theta'\in \mathcal{S}} p_{\theta \theta'}\int_{[0,\infty)}e^{-\gamma V(f(\theta,y+h,z),\theta')}\nu(dz)}
{\sum_{\theta'\in \mathcal{S}} p_{\theta \theta'}\int_{[0,\infty)}e^{-\gamma V(f(\theta,y,z),\theta')}\nu(dz)}\\
&= -\frac 1{h \gamma }\ln \sum_{\theta'\in \mathcal{S}} \int_{[0,\infty)}e^{-\gamma (V(f(\theta,y+h,z),\theta')-V(f(\theta,y,z),\theta'))}\nu'(\theta,y)(dz,\theta') \\
&\le \frac 1{h}\sum_{\theta'\in \mathcal{S}} \int_{[0,\infty)} \bigg(V(f(\theta,y+h,z),\theta')-V(f(\theta,y,z),\theta') \bigg) \nu'(\theta,y)(dz,\theta').
\end{align*}
Recall that, under the Assumptions (F1) and (F3), $f(\theta, y, z)$ is concave, non-decreasing and continuously differentiable, in $y$ for each $\theta \in {\cal S}$ and $z \in [0, \infty)$. Furthermore, $V(y, \theta)$ is concave and differentiable in $y$ by Lemma \ref{lem_env}, for each $\theta \in {\cal S}$. Hence using Mean Value Theorem we have
\begin{align*}
&\frac{V(f(\theta,y+h,z),\theta')-V(f(\theta,y,z),\theta')}{h}\\
&=\frac{V(f(\theta,y+h,z),\theta')-V(f(\theta,y,z),\theta')}{f(\theta,y+h,z)-f(\theta,y,z)}\frac{f(\theta,y+h,z)-f(\theta,y,z)}h \\
&\le  V'(f(\theta,y,z),\theta')f'(\theta,y,z).
\end{align*}
Thus, by combining above two inequalities, for each $h>0$,
$$ \frac{\widehat{V}(y+h,\theta)-\widehat{V}(y,\theta)}{h}\le \sum_{\theta'\in \mathcal{S}} \int_{[0,\infty)}  V'(f(\theta,y,z),\theta')f'(\theta,y,z)  \nu'(\theta,y)(dz,\theta')=: G(\theta,y).
$$
Hence for $y\in(0,\infty)$ and $\theta \in \mathcal{S}$,
\begin{equation}
\label{t21}
\widehat{V}'_+(y,\theta)\le G(\theta,y).
\end{equation}
Similarly for obtaining an estimate of the left-hand side derivative of $\widehat{V},$ we consider, for each $h > 0$,
\begin{align*}
& \frac{\widehat{V}(y-h,\theta)-\widehat{V}(y,\theta)}{-h} =
\frac 1{h \gamma }\ln \frac { \sum_{\theta'\in \mathcal{S}} p_{\theta \theta'}\int_{[0,\infty)}e^{-\gamma V(f(\theta,y-h,z),\theta')}\nu(dz)}
{\sum_{\theta'\in \mathcal{S}} p_{\theta \theta'}\int_{[0,\infty)}e^{-\gamma V(f(\theta,y,z),\theta')}\nu(dz)}\\
&= \frac 1{h\gamma}\ln \sum_{\theta'\in \mathcal{S}}  \int_{[0,\infty)}e^{-\gamma (V(f(\theta,y-h,z),\theta')-V(f(\theta,y,z),\theta'))}\nu'(\theta,y)(dz,\theta') \\
&\ge \sum_{\theta'\in \mathcal{S}} \int_{[0,\infty)}\frac{V(f(\theta,y-h,z),\theta')-V(f(\theta,y,z),\theta')}{-h} \nu'(\theta,y)(dz,\theta')\\
&= \sum_{\theta'\in \mathcal{S}}  \int_{[0,\infty)}\frac{V(f(\theta,y-h,z),\theta')-V(f(\theta,y,z),\theta')}{f(\theta,y-h,z)-f(\theta,y,z)}\frac{f(\theta,y-h,z)-f(\theta,y,z)}{-h} \nu'(\theta,y)(dz,\theta')\\
&\ge \sum_{\theta'\in \mathcal{S}} \int_{[0,\infty)} V'(f(\theta,y,z),\theta')f'(\theta,y,z) \nu'(\theta,y)(dz,\theta')= G(\theta,y).
\end{align*}
Hence, for $y\in(0,\infty)$ and $\theta \in \mathcal{S}$
\begin{equation}
\label{t22}
\widehat{V}'_-(y,\theta)\ge G(\theta,y).
\end{equation}
Next, we claim that, for each $\theta \in \mathcal{S}$, $G(\theta,\cdot)$ is continuous on $(0, \infty)$. To see this, first we observe that the integrand $g(\theta,y, z,\theta'):=  V'(f(\theta,y,z),\theta')f'(\theta,y,z)$ is continuous in $y$, thanks to Lemma \ref{lem_env}, Assumptions (F1) and (F3).
Again, due to the concavity of $f$ and $V$ (Assumption (F1) and part (a) of Theorem \ref{mainthm}) and the non-decreasing property of $f$ in $y$ (Assumption (F1)), $g(\theta,y, z,\theta')\le g(\theta,\epsilon, z,\theta')$ for all $y\ge \epsilon$, where $\epsilon$ is an arbitrarily small positive number, $z\ge 0$ and $\theta, \theta'\in \mathcal{S}$.
Again, since,
$$  \sum_{\theta'\in \mathcal{S}}\int_{[0,\infty)} g(\theta,\epsilon, z,\theta') \nu'(\theta,y)(dz,\theta') =  G(\theta,\epsilon)\le \widehat{V}'_-(\epsilon,\theta)<\infty,
$$
we may apply the dominated convergence theorem to conclude that $y\mapsto G(\theta,y)$ is continuous on $[\epsilon,\infty)$. As $\epsilon$ is an arbitrary positive number, $y\mapsto G(\theta,y)$ is continuous on $(0,\infty)$.\\

\noindent Since $\widehat{V}$ is concave and (\ref{t21}) and (\ref{t22}) hold, then for $h>0$ such that $h < y$, we obtain
\begin{equation}\label{g}
G(y+h,\theta)\le \widehat{V}'_-(y+h,\theta)\le \widehat{V}'_+(y,\theta)
\le G(\theta,y)\le \widehat{V}'_-(y,\theta)\le \widehat{V}'_+(y-h,\theta)\le G(y-h,\theta).
\end{equation}
Now letting $h\to 0^+$ in (\ref{g}),  continuity of $G(\cdot,\theta)$ yield that, for each $\theta \in \mathcal{S}$, $\widehat{V}(\cdot,\theta)$ is
continuously differentiable on $(0,\infty)$ and from the second statement of Lemma \ref{lem_env} we get
\begin{align*}
& \widehat{V}'({\phi}^*(x,\theta),\theta)  = G({\phi}^*(x,\theta),\theta) \\
& = \sum_{\theta'\in \mathcal{S}} \int_{[0,\infty)} V'(f(\theta,{\phi}^*(x,\theta),z),\theta')f'(\theta,{\phi}^*(x,\theta),z) \nu'(\theta,y)(dz,\theta') \\
& = \sum_{\theta'\in \mathcal{S}} \int_{[0,\infty)} u'(c^*({\phi}^*(x,\theta),\theta')) f'(\theta,{\phi}^*(x,\theta),z) \frac { p_{\theta \theta'} e^{-\gamma V(f(\theta,{\phi}^*(x,\theta),z),\theta')}}{\sum_{\theta'\in \mathcal{S}} p_{\theta \theta'} \int_{[0,\infty)}e^{-\gamma V(f(\theta,{\phi}^*(x,\theta),z),\theta')}\nu(dz)}  \nu(dz).
\end{align*}
Hence $\beta\widehat{V}'({\phi}^*(x,\theta),\theta)$ is identical to the right side of \eqref{e}. Again, as shown in Theorem \ref{mainthm}, and Lemma \ref{lem_kami} the policy ${\phi}^*(x,\theta) \in (0,x)$ is the maximizer of the map $(0,x)\ni y\mapsto u(x-y) +\beta \widehat{V}(y,\theta)$. Hence, the first-order condition with respect to $y$ gives that $ -u'(x-{\phi}^*(x,\theta))+\beta \widehat{V}'({\phi}^*(x,\theta),\theta) =0$. Thus, by part (b) of Theorem \ref{mainthm}, $\beta\widehat{V}'({\phi}^*(x,\theta),\theta)$ is also identical to the left side of \eqref{e}.
This completes the proof of part (a).\\

\noindent In order to prove (b) suppose that $x'< x''.$ We divide the proof in two sub-cases.

\noindent 1) If $x'=0,$ then by Lemma \ref{lem_kami} we have
${\phi}^*(x',\theta)=0<{\phi}^*(x'',\theta)$. Consequently, because $x''> x'=0$, $c^*(x'',\theta) = x''-{\phi}^*(x'',\theta)>0=c^*(x',\theta)$, where we have used the definition $c^* (x, \theta) := x - \phi^* (x, \theta)$ from Theorem \ref{mainthm}(b). Hence part (b) holds for this sub-case.

\noindent 2) We prove (b) for the sub-case $x'' > x' >0$ by contradiction. Since on the right side of the Euler equation \eqref{e} the variable $x$ appears via ${\phi}^*(x,\theta)$ only, we obtain  $u'(x'-{\phi}^*(x',\theta))=u'(x''-{\phi}^*(x'',\theta))$, if we assume ${\phi}^*(x',\theta)={\phi}^*(x'',\theta)$. But the equality cannot hold,
since  $u$ is strictly concave. Similarly, if
$c^*(x',\theta)=c^*(x'',\theta),$ then by the second statement of Lemma \ref{lem_env} we must have, for each $\theta \in \mathcal{S}$,
$V'(x',\theta)=u'(c^*(x',\theta))=u'(c^*(x'',\theta))=V'(x'',\theta).$
However, this equality contradicts the strict concavity of $V(\cdot,\theta)$. Recall from Theorem \ref{mainthm}(b) that the functions $\phi^* (x, \theta)$ and $c^* (x, \theta)$ are non-decreasing in $x$, for each $\theta \in {\cal S}$. Therefore, (b) holds true for every sub-case.
\end{proof}

\section{Stationary Distributions}\label{sec:stationaryDist}
\noindent  In this section, a dynamics of the growth is obtained when the economic agent follows a stationary policy. This along with the regime switching dynamics jointly become a discrete time Markov chain on product state space $[0, \infty) \times {\cal S}$. Consider this Markov chain when the optimal stationary policy ${\phi}^*\in \Phi$ is followed by the economic agent. We recollect and redefine
\begin{align} \label{optimal}
\left.\begin{array}{rl}
 \mathbb{P}(\theta_{k+1}=j\mid h_k)= & \mathbb{P}(\theta_{k+1}=j\mid \theta_k)=p_{\theta_k j}\\
 x_{k+1} = & f(\theta_k, {\phi}^*(x_k,\theta_k),\xi_{k}),
\end{array}\right\}
\end{align}
where $\{\xi_k\}_{k\in\mathbb{N}}$  is independent of $\{\theta_k\}_{k\in\mathbb{N}}$ and i.i.d. random sequence, distributed as $\nu$, a probability measure on $[0,\infty)$.
Clearly, $ \{(x_k,  \theta_k)\}_{k \in \mathbb{N}}$ is a Markov chain with respect to the minimal sigma-field generated by $\{h_k\}_{k \in \mathbb{N}}$. We also recall that due to (F5), $f(\theta,y,z)>0$ for every  $\theta \in \mathcal{S}$, $y\in(0,\infty)$, and in $z$, $\nu$-a.e.. In addition to this, utilising the fact that ${\phi}^*(x,\theta)\in(0,x)$ for $x>0$ and $\theta \in \mathcal{S}$, and the Assumption (F1) that $f(\theta, y, z)$ is non-decreasing in $y$ for every $\theta \in {\cal S}$ and $z \in [0, \infty)$, we may confine ourselves to the study of the income process on $(0,\infty)$.\\%

\noindent In this section, we shall show the existence of at least one non-trivial stationary distribution of \eqref{optimal}. To this end we first show that the system is globally stable under a set of assumptions. Next, following the approach of \cite{bjjet} and references therein, we use the Euler equation \eqref{e} and the Foster-Lyapunov theory (see \cite{f,mt}) of Markov chains for showing the existence of a non-trivial  stationary distribution.

We impose the following assumptions on the distribution of the random shock $\nu$ and the production function $f$ for pursuing the subsequent investigation.

\begin{enumerate}
  \item[(D1)]  $\displaystyle{ \lim_{y\to 0^+} \max_{\theta \in \mathcal{S}} \int_{[0,\infty)} \frac1{\beta f'(\theta,y,z)}\nu(dz)<1.}$
  \item[(D2)] There exist $\lambda_2\in(0,1)$ and $\kappa_2>0$ such that
$$ \max_{\theta \in \mathcal{S}} \int_{[0,\infty)} f(\theta,y,z)\nu(dz)\le \lambda_2 y+\kappa_2,\quad y\in[0,\infty).$$
\item[(D3)] The transition probability matrix $p
= (p_{\theta \theta'})_{\mathcal{S}\times \mathcal{S}}$ is irreducible.
 \end{enumerate}

\noindent
The term $f'(\theta,y,z)$ represents the production rate at $y$ when shock is $z$ and state is $\theta$. This value is nonzero $\nu$-a.e. due to (F5). Thus the integrand in Assumption (D1) is well defined.  Clearly, (D1) precludes the situation where on an average the ratio of growth rates of ideal bank and the production is larger than one at every economic state. On the other hand, the role of Assumption (D2) is to preclude the situation where the average production rate is infinity.
Moreover, since the state space $\mathcal{S}$ is finite, (D3) ensures existence of a unique stationary distribution of the Markov chain $\{\theta_k\}$, see \cite[Chapter 1]{s}.

\begin{lem} \label{lem_distr}
Assume that (D1) holds. Then, for $W_1(x,\theta):= \sqrt{ u'(c^*(x,\theta))e^{-\gamma V(x,\theta)}},$ $x\in(0,\infty)$ and $\theta \in \mathcal{S}$, there exist
$\lambda_1\in(0,1)$ and $\kappa_1>0$ such that
$$\sum_{\theta' \in \mathcal{S}} p_{\theta \theta'} \int_{[0,\infty)} W_1(f(\theta,{\phi}^*(x,\theta),z),\theta')\nu(dz)\le \lambda_1 W_1(x,\theta)+\kappa_1,\quad x\in(0,\infty), \theta \in \mathcal{S}.$$
\end{lem}

\begin{proof} First note that $W_1 (x, \theta)$ is well-defined since by the Assumption (U1), the utility function $u$ is increasing.
Next, by the Cauchy-Schwarz inequality, it follows that for each $x\in(0,\infty), \theta \in \mathcal{S}$
\begin{align}
\label{cs}\nonumber
\sum_{\theta' \in \mathcal{S}}& p_{\theta \theta'} \int_{[0,\infty)} W_1(f(\theta,{\phi}^*(x,\theta),z), \theta')\nu(dz)\\ \nonumber
=&\sum_{\theta' \in \mathcal{S}} p_{\theta \theta'} \int_{[0,\infty)} \left[
u'(c^*(f(\theta,{\phi}^*(x,\theta),z),\theta')) e^{-\gamma V(f(\theta,{\phi}^*(x,\theta),z),\theta')} \right]^{\frac12}\nu(dz)\\ \nonumber
=&\sum_{\theta' \in \mathcal{S}} p_{\theta \theta'} \int_{[0,\infty)} \left[ u'(c^*(f(\theta, {\phi}^*(x,\theta),z),\theta')) e^{-\gamma V(f(\theta,{\phi}^*(x,\theta),z),\theta')} \frac{\beta f'(\theta,{\phi}^*(x,\theta),z)}{\beta f'(\theta,{\phi}^*(x,\theta),z)} \right. \\\nonumber
& \hspace{2in } \times \left. \frac { \sum_{\theta'\in \mathcal{S}} p_{\theta \theta'} \int_{ [0,\infty)}e^{-\gamma V(f(\theta,{\phi}^*(x,\theta),z),\theta')}\nu(dz)} { \sum_{\theta'\in \mathcal{S}} p_{\theta \theta'} \int_{ [0,\infty)}e^{-\gamma V(f(\theta,{\phi}^*(x,\theta),z),\theta')} \nu(dz)}\right]^{\frac12}\nu(dz)\\ \nonumber
\le& \left[\sum_{\theta' \in \mathcal{S}} p_{\theta \theta'} \int_{[0,\infty)} u'(c^*(f(\theta,{\phi}^*(x,\theta),z),\theta'))  \frac{e^{-\gamma V(f(\theta, {\phi}^*(x,\theta),z), \theta')} \beta f'(\theta,{\phi}^*(x,\theta),z)} { \sum_{\theta'\in \mathcal{S}} p_{\theta \theta'}\int_{ [0,\infty)}e^{-\gamma V(f(\theta,{\phi}^*(x,\theta),z),\theta')}\nu(dz)}\nu(dz)\right]^{\frac12}\\
&\hspace{1.7in} \times \left[  \int_{ [0,\infty)} \frac { \sum_{\theta' \in \mathcal{S}} p_{\theta \theta'} \int_{ [0,\infty)}e^{-\gamma V(f(\theta,{\phi}^*(x,\theta),z),\theta')} \nu(dz)} {\beta f'(\theta,{\phi}^*(x,\theta),z)}\nu(dz)\right]^{\frac12}.
\end{align}
Furthermore,  substitution of   \eqref{e} into \eqref{cs}  yields
\begin{align}
\label{cs1}\nonumber
\sum_{\theta' \in \mathcal{S}}& p_{\theta \theta'} \int_{[0,\infty)} W_1(f(\theta,{\phi}^*(x,\theta),z),\theta')\nu(dz)\\ \nonumber
\le& \sqrt{u'(c^*(x,\theta))}\left[
\int_{[0,\infty)}\frac{ \nu(dz)} {\beta f'(\theta,{\phi}^*(x,\theta),z)} \sum_{\theta' \in \mathcal{S}} p_{\theta \theta'}
 \int_{ [0,\infty)}e^{-\gamma V(f(\theta,{\phi}^*(x,\theta),z),\theta')} \nu(dz)\right]^{\frac12}\\
\le& \sqrt{u'(c^*(x,\theta))}\left[
\int_{[0,\infty)}\frac{ \nu(dz)} {\beta f'(\theta,x,z)} \right]^{\frac12},
\end{align}
as $V (\cdot,\theta)\ge 0$, and $f'(\theta,x,z) \le f'(\theta, {\phi}^*(x,\theta),z)$ for each $\theta\in \mathcal{S}$, $x>0$ due to  (F1) and ${\phi}^*(x,\theta) < x$. Now from Assumption (D1) it follows that there exists $\eta<1$ and $\delta'>0$ such that $\int_{[0,\infty)} (\beta f'(\theta,\delta,z))^{-1}\nu(dz) < \eta$ for every $\delta\le \delta'$ and $\theta \in \mathcal{S}$. Again as $V(0,\theta)=0$, there is a $\delta\in (0,\delta']$ such that $\exp(\gamma V(\delta,\theta)/2)<1/\sqrt{\eta}$. Hence
$$\lambda_1:= \max_{\theta \in \mathcal{S}} e^{\frac{\gamma V(\delta ,\theta)}{2}}\left[\int_{[0,\infty)}\frac{ \nu(dz)} {\beta f'(\theta,\delta,z)}\right]^{\frac12}<1$$
where $\delta$ is chosen as mentioned above. Since $V$ is non-decreasing (as $V \in \mathcal{B}$ due to Theorem \ref{mainthm}(a)), for $x\in(0,\delta)$  we have $e^{-\frac{\gamma V(x ,\theta)}{2}} e^{\frac{\gamma V(\delta ,\theta)}{2}} \ge 1.$
Consequently, by \eqref{cs1}
\begin{align}\label{81} \nonumber
\sum_{\theta' \in \mathcal{S}}& p_{\theta \theta'} \int_{[0,\infty)} W_1(f(\theta,{\phi}^*(x,\theta),z),\theta')\nu(dz)\\
  \le & \sqrt{u'(c^*(x,\theta))e^{-\gamma V(x,\theta)}} e^{\frac{\gamma V(\delta,\theta)}{2}}\left[
\int_{[0,\infty)}\frac{ \nu(dz)} {\beta f'(\theta,\delta,z)} \right]^{\frac12}= \lambda_1 W_1(x,\theta).
\end{align}
For $x\ge \delta$ and $\theta \in \mathcal{S}$, since $\phi^*(\delta, \theta) \in (0, \delta)$, the functions  $V(\cdot,\theta)$, $c^*(\cdot,\theta)$, $f(\theta,\cdot,z)$ are non-decreasing, and $u'$, $1/\exp(\cdot)$ are non-increasing, we have for each $\theta \in \mathcal{S}$
\begin{align}
\nonumber
\sum_{\theta' \in \mathcal{S}}& p_{\theta \theta'} \int_{[0,\infty)} W_1(f(\theta,{\phi}^*(x,\theta),z),\theta')\nu(dz)\\
\nonumber \le & \sum_{\theta' \in \mathcal{S}} p_{\theta \theta'} \int_{[0,\infty)}\left[ u'(c^*(f(\theta,{\phi}^*(\delta,\theta),z),\theta')) e^{-\gamma
  V(f(\theta,{\phi}^*(\delta,\theta),z),\theta')}\right]^{\frac 12}\nu(dz).
\end{align}
Using \eqref{cs1}, the right side of above is bounded above by
\begin{align}
\label{82}& \sqrt{u'(c^*(\delta,\theta))}\left[
\int_{[0,\infty)}\frac{ \nu(dz)} {\beta f'(\theta,\delta,z)} \right]^{\frac12} \le \max_{\theta \in \mathcal{S}} \sqrt{u'(c^*(\delta,\theta))}\left[
\int_{[0,\infty)}\frac{ \nu(dz)} {\beta f'(\theta,\delta,z)} \right]^{\frac12} =:\kappa_1.
\end{align}
The result follows from the estimates (\ref{81}) and (\ref{82}) combined together.
\end{proof}
\noindent Let us denote by $2^{\mathcal{S}}$ the power set of ${\cal S}$. For a given $\theta \in \mathcal{S},x \in [0,\infty), \Theta \in 2^{\mathcal{S}}, A \in \mathcal{B}([0,\infty))$, the measurable map $Q$, given by
 \begin{align}\label{transitionKernel}
     Q(\theta,x,\Theta, A) & := \sum_{\theta' \in \mathcal{S}} p_{\theta \theta'}\int_{[0,\infty)} \mathbbm{1}_A(f(\theta, \phi^*(x,\theta),z)) \mathbbm{1}_{\Theta} (\theta')\nu(dz) \nonumber \\ 
     & = \sum_{\theta' \in \Theta} p_{\theta \theta'} \nu[z \in [0,\infty): f(\theta, \phi^*(x,\theta),z) \in A].
 \end{align}
Note that due to the Theorem 8.9 in \cite{slp} the above defined $Q$ is the stochastic kernel of the Markov chain \eqref{optimal} on $(\mathcal{S}\times [0,\infty); 2^{\mathcal{S}} \times \mathcal{B}\times([0,\infty)))$ in our setting.
In other words, the number $Q(\theta,x,\Theta, A)$ is the probability that the economic system will move from state $(x,\theta)$ to some state in the set $\Theta \times A$ after one period of time.
In the remaining part of this section we aim to show that $Q$ is globally stable. To this end we need to establish some properties of this stochastic kernel. A stochastic kernel $Q$ defines a linear operator $T$ from bounded measurable functions, on $\mathcal{S} \times [0,\infty)$, to itself via the formula
\begin{align*}
(T\psi)(\theta,x) = \sum_{\theta' \in \mathcal{S}}\int_{[0,\infty)} \psi(\theta',y) Q(\theta, x, \theta', dy).
\end{align*}
The kernel $Q$ is said to have the Feller property if, for each $\theta \in \mathcal{S}$, $(T \psi)(\theta,\cdot)$ is bounded and continuous whenever $\psi(\theta,\cdot)$ is, see \cite{cw} for details. Since
\begin{align*}
(T\psi)(\theta,x)= & E[ \psi(\theta_{k+1}, f(\theta_k, \phi^*(x_k,\theta_k),\xi_k)) \mid x_k=x, \theta_k=\theta]\\
= &\sum_{\theta' \in \mathcal{S}}\int_{[0,\infty)}p_{\theta \theta'} \psi(\theta', f(\theta, \phi^*(x,\theta),z)) \nu(dz),
\end{align*}
due to the boundedness of $\psi$ and continuity of $x\mapsto \psi(\theta', f(\theta, \phi^*(x,\theta),z))$, direct application of dominated convergence theorem gives continuity of $x\mapsto (T\psi)(\theta,x)$. Hence $Q$ defined in \eqref{transitionKernel} is Feller.

\noindent Let us denote the space of measures on $[0, \infty)$ by $Pr([0,\infty))$. A sequence $\{\mu_n\}_{n \in \mathbb{N}} \subset Pr([0,\infty))$ is called tight if, for all $\varepsilon > 0$, there exists a compact set $K \subset [0,\infty)$ such that $\mu_n( [0,\infty) \setminus K) \leq  \varepsilon$ for all $n$, refer \cite{b} for details.  For given $\theta \in \mathcal{S}, x \in [0,\infty), \Theta \in 2^{\mathcal{S}}$ and  $A \in \mathcal{B}([0,\infty))$ we set  $Q^1 := Q$ and for $k>1$
\begin{align}\label{Qk}
    Q^k (\theta,x,\Theta, A) :=\sum_{\theta' \in \mathcal{S}} \int_{[0,\infty)} Q^{k-1}(\theta',y,\Theta,A) Q(\theta,x,\theta',dy).
\end{align}
A stochastic kernel $Q$ is said to be bounded in probability if
the sequence $\{ Q^k(\theta,x, \Theta, \cdot)\}_{k\in \mathbb{N}}$ is tight for all $x \in [0,\infty)$, $\theta \in \mathcal{S}$ and $\Theta \in 2^{\mathcal{S}}$.

\begin{lem} \label{lem.last}For each $\theta \in \mathcal{S}$ and $x \in [0,\infty)$ the sequence $\{ Q^k(\theta,x,\cdot,\cdot)\}_{k\geq 0}$, as in \eqref{Qk}, is bounded in probability.
\end{lem}
\begin{proof}
We first note that the map $W:(0,\infty)\times\mathcal{S} \to \mathbb{R}$ given by
$$W(x,\theta)=W_1(x,\theta) + x, \qquad x\in(0,\infty)$$
is coercive where $W_1 (\theta, x)$ is defined in Lemma \ref{lem_distr}. Indeed,
$$\inf_{x\notin [1/n,n]}W(x,\theta)\ge \min\left(n, \sqrt{ u'(c^*(1/n,\theta))e^{-\gamma V(1/n,\theta)}}\right).$$
As $c^*(0+,\theta)=0$, $u'(0+)=\infty$, and $V(0+,\theta)=0$, from above, $\lim_{n\to \infty} \inf_{x\notin [1/n,n]}W(x,\theta) =\infty$. Next we show that 
for all $x\ge 0$,  $\theta \in \mathcal{S}$
\begin{align}\label{kappa}
\int_{[0,\infty)} \sum_{\theta' \in \mathcal{S}} p_{\theta \theta'}W(f(\theta,{\phi}^*(x,\theta),z),\theta')\nu(dz)\le \lambda W(x,\theta)+\kappa
\end{align}
with $\lambda:=\max\{\lambda_1,\lambda_2\}$  and $\kappa:=\kappa_1+\kappa_2$, where $\lambda_1$ and $\kappa_1$ are as in Lemma \ref{lem_distr}, and $\lambda_2$ and $\kappa_2$ are as in (D2)
Indeed, since ${\phi}^*(x,\theta) < x$, using Lemma \ref{lem_distr} and Assumption (D2),
\begin{align*}
    \sum_{\theta' \in \mathcal{S}}& p_{\theta \theta'} \int_{[0,\infty)} W(f(\theta,{\phi}^*(x,\theta),z),\theta')\nu(dz) \\
    = & \sum_{\theta' \in \mathcal{S}} p_{\theta \theta'} \int_{[0,\infty)} W_1(f(\theta,{\phi}^*(x,\theta),z),\theta')\nu(dz)  + \int_{[0,\infty)} f(\theta,{\phi}^*(x,\theta),z)\nu(dz) \\
    \le &\lambda_1 W_1(x,\theta) + \kappa_1 + \lambda_2 {\phi}^*(x,\theta) + \kappa_2\\
    \le& \lambda_1 W_1(x,\theta) + \kappa_1 + \lambda_2 x + \kappa_2 \le \lambda W(x,\theta) + \kappa.
\end{align*}
Hence, using \eqref{kappa} we get
\begin{align*}
E[W(x_k,\theta_k)\mid x_0=x, \theta_0=\theta]
= & E[E[W(x_k,\theta_k)\mid x_{k-1}, \theta_{k-1}]\mid x_0=x, \theta_0=\theta]\\
\le & \lambda E[ W(x_{k-1}, \theta_{k-1}) \mid x_0=x, \theta_0=\theta] + \kappa.
\end{align*}
By applying the above inequality repeatedly, we get
$$E[W(x_k,\theta_k)\mid x_0=x, \theta_0=\theta] \le  \lambda^k W(x,\theta) + \kappa\frac{1-\lambda^k}{1-\lambda}.$$
Thus we have $\limsup_{k\to \infty} E[W(x_k,\theta_k)\mid x_0=x, \theta_0=\theta] <\infty$ as $\lambda\in (0,1)$.
Now we apply Lemma D.5.3 of \cite{mt} to conclude  that $\{(x_k,\theta_k)\}$ is bounded in probability.
\end{proof}

\begin{thm} \label{Theo3} Assume that  (U1)-(U4), (F1)-(F5) and (D1)-(D3) hold. Then there exists a non-trivial stationary distribution for the Markov chain in \eqref{optimal}.
\end{thm}

\begin{proof}
Due to (D3) the finite-state Markov chain $\{\theta_k\}_k$ has a unique stationary distribution $\mu$, say. Let $\theta_0$ be sampled from $\mu$. Then using Lemma \ref{lem.last} and the Feller property of $Q$, the theorem follows as a direct application of Theorem 12.0.1 of \cite{mt}.
\end{proof}

\section{Numerical Experiment}

\noindent In this section we show that the natural regime switching extension of a model with Cobb–Douglas production function coupled with a power utility would satisfy the assumptions presented in Sections \ref{sec:optimization} - \ref{sec:stationaryDist}. Further, we provide numerical results to illustrate a numerical approximation to the value function in Eq. \eqref{oe} under a parametric model to be considered in Section 7.1 below and the impacts of introducing regime switches on the (approximate) value function and the optimal investment ratio as the income varies. Specifically, the (approximate) value function and the optimal investment ratio will be compared with their counterparts under the non-regime-switching model studied in \cite[Example 1]{bjjet}.

\subsection{A parametric model}
We consider the following example of regime switching extension of a model which has Cobb–Douglas production function
$$f(\theta,y,z) = y^{\omega(\theta)}z,$$
where $\omega$ is a given function defined on a finite set $\mathcal{S}$ taking values in $(0,1)$.
Thus, it is clear that (F1) holds. Note that the income process in this case evolves as
 $$x_{k+1}= y_k^{\omega(\theta_k)}\xi_k, \quad k\in\mathbb{N},$$
where $\{\theta_k\}$ is an irreducible Markov chain with state space ${\cal S}$ and transition probability matrix $(p_{ij})$. Hence (D3) is true. As in \cite[Example 1]{bjjet}, we consider the power utility, i.e., $u(a)=a^\sigma$ with $\sigma\in(0,1)$. Therefore, (U1) holds. Assumption (U2) holds for $w(x,\theta)=(r+x)^\sigma,$ for any constant $r\ge 1$ whose more suitable range will be set later for satisfying other assumptions. Furthermore, the random shocks and their reciprocals are assumed to have finite mean, i.e.,  $\bar{z}:=\int_{[0,\infty)}z\nu(dz)<\infty$ and $\int_{[0,\infty)}z^{-1} \nu(dz)<\infty$. For example, lognormal distributions possess this.

Next we check Assumption (F2). Let us fix the discount factor $\beta \in (0,1)$ and income of an agent to $x$. Then, for $\theta \in \mathcal{S}$, we have,
\begin{align}
    \sup_{y \in [0,x]} \sup_{\theta^\prime \in \mathcal{S}} \int_{[0,\infty)} w(f(\theta,y,z),\theta^\prime)\, \nu(dz)  
    & = \sup_{y \in [0,x]} \sup_{\theta^\prime \in \mathcal{S}} \int_{[0,\infty)} \left(r+ y^{\omega(\theta)}z  \right)^\sigma\nu(dz) \nonumber\\
    & \leq \sup_{y \in [0,x]} \sup_{\theta^\prime \in \mathcal{S}} \left(r+ y^{\omega(\theta)}\int_{[0,\infty)} z \nu(dz) \right)^\sigma \nonumber\\
    & = (r+ x^{\omega(\theta)} \bar{z})^\sigma.
\end{align}
Here the first inequality follows from the Jensen inequality for concave function ($x\mapsto w(x,\theta)$ for each $\theta\in \mathcal{S}$). So, for given $(x,\theta)$, by dividing $w(x,\theta)$ to both sides we get
\begin{align}\label{ex1-2}
    & \frac{\sup_{y \in [0,x]} \sup_{\theta^\prime \in \mathcal{S}} \int_{[0,\infty)} w(f(\theta,y,z),\theta^\prime)\, \nu(dz) } {w(x,\theta)}  \leq  \left( \frac{r+ x^{\omega(\theta)} \bar{z}}{r+x} \right)^\sigma.
\end{align}
Since $\frac{r+ x^{\omega(\theta)} \bar{z}}{r+x} \leq 1 $ iff  $\bar{z}\leq  x^{1-\omega(\theta)}$,  we consider two cases.
If $x \geq \bar{x}:= \bar{z}^{\frac{1}{1-\omega(\theta)}}$, then the right side of \eqref{ex1-2} is bounded above by $1$.  Since $\beta <1$, (F2)  holds true for every $\alpha\in (1,1/\beta)$ whenever $x \geq \bar{x}$. In the second case we consider the complementary region of income $x$, i.e., $0\le x < \bar{x}$. Then
\begin{align}
    & \left( \frac{r+ x^{\omega(\theta)} \bar{z}}{r+x} \right)^\sigma \leq \left( \frac{r+ x^{\omega(\theta)} \bar{z}}{r} \right)^\sigma \leq  \left( 1 + \frac{\bar{x}^{\omega(\theta)} \bar{z}}{r} \right)^\sigma  = \left( 1 + \frac{\bar{x}}{r} \right)^\sigma = \left( 1+\frac{ (\bar{z}) ^{\frac 1{1-\omega(\theta)}}}r \right)^\sigma.
\end{align}
Hence, for given $\beta \in (0,1),$ one can choose $r\geq 1$ such that  $\alpha := \left( 1 + \frac{\bar{x}}{r} \right)^\sigma $ satisfies  $\alpha \beta <1$. Thus for this choice of $w$ and $\alpha$, (F2) holds for all $x\ge 0$.
It is straightforward to see that the Assumptions (U3)-(U4) and (F3)-(F5) also hold true.

The only remaining Assumptions (D1) and (D2) are examined below. To establish (D1), w.l.o.g. we can assume that $y \in (0,1)$. Then, since $f^\prime(\theta,y,z) = \omega(\theta) y^{\omega(\theta)-1}z$ and $y^{\omega(\theta)-1} \geq 1$ for $y \in (0,1)$,
$$ \frac{1}{\beta f^\prime(\theta,y,z)} \le \frac{1}{\beta \omega(\theta) z} .$$
But, due to our assumption of the finiteness of $\int_{[0,\infty)}z^{-1} \nu(dz)$ the right side expression of the second inequality above is integrable w.r.t. $\nu$. Hence, the Dominated Convergence Theorem is applicable to $ \frac{1}{\beta f^\prime(\theta,y,z)}$. Thus by noting that $\omega (\theta) < 1$ for each $\theta \in {\cal S}$,
$$\lim_{y \to 0+} \max_{\theta \in \mathcal{S}}\int_{[0,\infty)}\frac{1}{\beta f^\prime(\theta,y,z)} \nu(dz)= \max_{\theta \in \mathcal{S}} \int_{[0,\infty)} \lim_{y \to 0+} \frac{1}{\beta (\omega(\theta) z y^{\omega(\theta)-1} )} \nu(dz) =0.$$
Therefore, (D1) is true. Using the fact that for any $0<a<1$,  $y^a \le cy +  (1-a) (a/c)^{\frac{a}{1-a}})$ for all $c>0$,
\begin{align*}
\int_{[0,\infty)} f(\theta,y,z) \nu (dz)
= & y^{\omega(\theta)}\bar{z}
\le \frac{\bar{z}}{1+\bar{z}} y + \bar{z} (1-{\omega(\theta)}) \left(\omega(\theta)+\omega(\theta)\bar{z}\right)^{\frac{\omega(\theta)}{1-\omega(\theta)}}.
\end{align*}
Thus (D2) holds with $\lambda_2 = \frac{\bar{z} }{1+\bar{z}} \in (0,1)$ and  $\kappa_2=\max_{\theta \in \cal{S}} \bar{z} (1-{\omega(\theta)}) \left((1+\bar{z}){\omega(\theta)}\right)^{\frac{\omega(\theta)}{1-\omega(\theta)}} >0$.

\subsection{Numerical values of parameters}
The hypothetical values of the parameters are specified in this subsection. Specifically, we set $\beta = 0.9$ and the risk sensitivity parameter $\gamma = 1$. The distribution of the random shock $\nu$ is assumed to be the probability distribution function of a standard Lognormal distribution. A three-state Markov chain is considered, and the state space ${\cal S}$ of the chain is assumed to be $\{1,2,3\}$. The transition probability matrix and the parameters $\omega$ are taken as
$$p= \left(
\begin{array}{ccc}
0.50 & 0.40 & 0.10  \\
0.25 & 0.50 & 0.25\\
0.10 & 0.40 & 0.50
\end{array}\right)
\textrm{ and }
\begin{array}{r|ccc}
\theta   &  1 &  2 &  3  \\ \hline
\omega & 0.3& 0.5& 0.9
\end{array}.
$$
From the specified numerical values for $\omega$ in the three different states, the regimes $\theta=1$ and $\theta=3$ are associated with the low and high production rate regimes, respectively provided the investment value is not too small. The power parameter $\sigma$ of utility is taken as $1/2$.

\subsection{Computation of optimal investment}
A numerical approximation of the solution to \eqref{oe} for the parametric model as described in Subsections 7.1 and 7.2 is presented below. The proposed numerical method relies on the fact that the solution of \eqref{oe} is the fixed point of the contraction $L$ (as in \eqref{oper_1}) which is also the limit of repeated application of the operator $L$ on a fixed function in the Banach space. Starting with the zero function, we stop after third step of iteration to obtain a numerical approximation of the value function. More specific details are described below.

\noindent Under the assumption of $\xi_k \sim \textrm{Lognormal}(0,1)$, the integral w.r.t. $\nu$ in \eqref{oe}, is simply $$\int_{[0,\infty)} e^{-\gamma V(f(\theta,y,z),\theta')}dF(z)$$ where $F$ is the cdf of the standard lognormal distribution. This can be rewritten as $\int_{0}^{1} e^{-\gamma V(f(\theta,y,F^{-1}(t)),\theta')}dt$. This is numerically approximated using the composite trapezoidal rule by considering 18 equally spaced consecutive sub-intervals of $[0,1]$. The supremum w.r.t. $y$ on $[0,x]$ is approximated by the maximum of values obtained for 30 equi-spaced values of $y$ in $[0,x]$\footnote{The running time of the code in online Matlab compiler is nearly two minutes.}.

\begin{minipage}{\linewidth}
	\centering
	\begin{minipage}{0.45\linewidth}
        \begin{figure}[H]
          	\includegraphics[width=\linewidth]{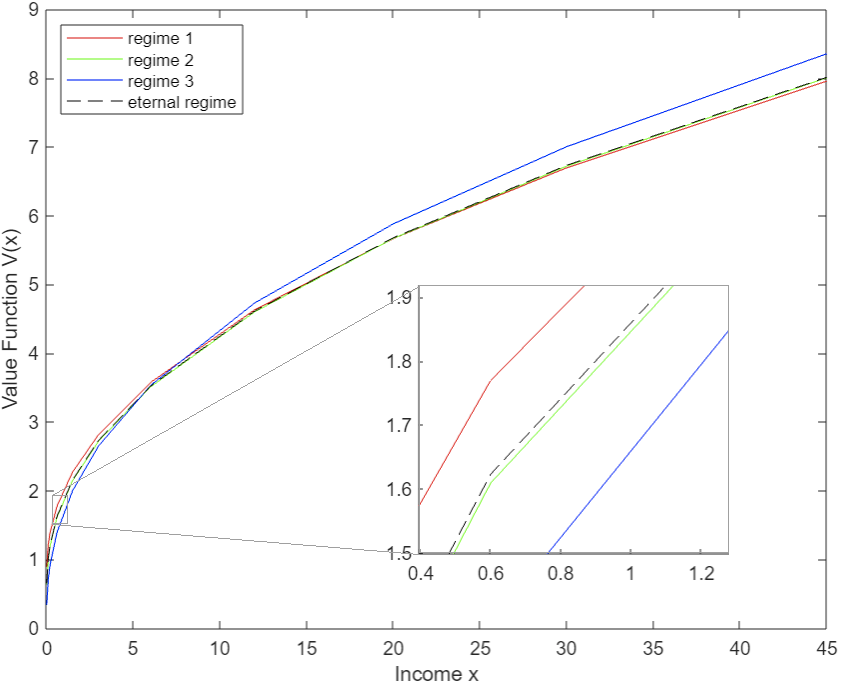}
            \caption{Value function: Numerical solution to \eqref{oe}} \label{valuefn}
        \end{figure}
	\end{minipage}
	\hspace{0.02\linewidth}
	\begin{minipage}{0.45\linewidth}
        \begin{figure}[H]
            \includegraphics[width=\linewidth]{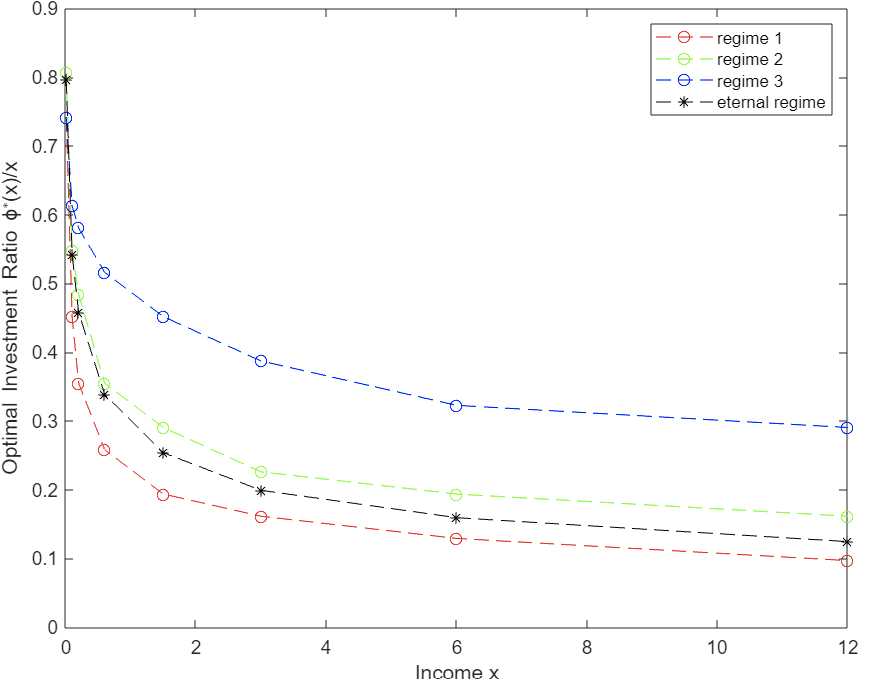}
            \caption{Optimal Investment Ratio: Numerical solution of \eqref{oe_max}} \label{oi}
        \end{figure}
	\end{minipage}
\end{minipage}

In Figure \ref{valuefn} the approximate value function $V(x, \theta)$ is plotted along vertical axis against some specific $x$ values along horizontal axis for each value of $\theta \in \{1,2,3\}$. The line plots for $V(x,1)$, $V(x,2)$ and $V(x,3)$ are in red, green and blue colors respectively. The black broken line plot corresponds to the value function of a single-regime model with parameter values identical to those of regime 2. Similar color scheme is adopted in Figure \ref{oi}, where the ratio of optimal investment and income (i.e., ${\phi}^*(x,\theta)/x$) is plotted against income for each regime and eternal middle regime cases. The Matlab code, used in this paper, can be
accessed from \emph{https://github.com/agiiser/Regime-Switching-Optimal-Growth-Model}.

\subsection{Interpretation}
Figure \ref{oi}, where the optimal investment ratios are plotted, illustrates an intuitive phenomena, namely, the investment ratios depend on the production rates. Specifically, from Figure \ref{oi}, it can be seen that the optimal investment ratio curve shifts up when $\theta$ increases from 1 to 3 through 2, (i.e., the production rate increases from the lowest to the highest). This is consistent with the intuition that it may be optimal to invest more when the production rate is higher. It is also noted that for each fixed production rate regime $\theta$, the optimal investment ratio decreases as the income increases. This reflects that the ratio of optimal consumption and income increases with hike of income. These are common features observed in the fixed/eternal regime models too. Now we discuss some other features which are specific to regime switches in the production rate only.

We first note that the stationary distribution of the regime switching dynamics in the numerical example is $(\frac{5}{18}, \frac{8}{18}, \frac{5}{18})$. This represents a typical scenario where the middle state is more likely than the extreme states on both sides. Furthermore, the transition to low or high production from the middle state is equally likely. Despite this symmetry in transition, influence of an extreme regime on the middle regime can be considered in the following scenario. In principle, if the production function parameters in low and high regimes are not similarly far from the middle regime, the one which is more extreme, should influence the investment decision at the middle regime too. This is verified with the numerical example under consideration. We note that the difference of the values of $\omega(2)$ and   $\omega(3)$ is much larger than the difference of the values of $\omega(2)$ and  $\omega(1)$. As a result the influence of a higher production rate regime should also be observed in the middle production rate regime. This is indeed observed in Figure \ref{oi}. The green line, for middle regime is found above the black line, which corresponds to the eternal middle regime with no chance of switching. Such influence is also seen in Figure \ref{valuefn}. The value function at the middle regime deviates from that for eternal middle regime and moves towards the value function at third regime. The value function plots also exhibit standard features like monotonicity and concavity w.r.t. income. Also, for significantly large incomes, the value function shifts up when the production rate increases from the lowest regime to the highest regime. However, due to the very defining formula, such association of production function is reversed when investment is sufficiently small, which is the case when income is low. This explains the visible inflection of ordering of value function w.r.t. the regime order.

\section{Conclusion}
\noindent One can perhaps imagine a formulation of regime switching dynamics, that is different from ours, where
$$h_k=\left\{\begin{array}{l@{\quad \quad}l}
x_1,& \mbox{ for } k=1,\\
(x_1, y_1, \theta_1, x_2, y_2,\theta_2,\ldots, x_{k-1}, y_{k-1}, \theta_{k-1}, x_k), & \mbox{ for } k\ge2,
 \end{array}\right.$$
denotes the history at $k$th time unit. In this case, the state of the market is only observed on the next day. This would also imply that the production function depends on two different sources of randomness, one described by i.i.d. ($\xi$) and another auto-correlated ($\theta$). However, a stationary strategy in this setting depends solely on the income $x$ of the day as in \cite{bjjet}.
In yet another formulation, the market states may be unobserved but influences the production, and the scope of the policies are restricted to be those which are insensitive to the past market states. Then the past realizations of $\theta$ do not appear in the history process. This one appears as a minor generalization of the literature, where independent noise $\xi_k$ is replaced by a correlated noise $(\xi_k,\theta_k)$ only and rest are identical.
Apart from the above mentioned settings, the formulation of an incomplete information setting is also interesting where, $\theta$ is partially observed only via the realization of production function, and the optimization is not restricted to the state-insensitive policies.\\

\noindent \textbf{Acknowledgement:} We would like to sincerely thank Professor John Stachurski for insightful and valuable comments.

\begin{appendix}
\section{}
\begin{proof}[Proof of Theorem \ref{P}] We note that the constant zero function belongs to $B_w(H_{k+1})$. Furthermore,
$\rho_{\pi_k,h_k}(0) = -\frac{1}{\gamma} \ln \sum_{\theta'\in\mathcal{S}} p_{\theta_k \theta'} \int_{[0,\infty)} \nu(dz)=-\frac{\ln 1}{\gamma}=0$. Using monotonicity of expectation and reverse monotonicity of $x\mapsto e^{-x}$ and $x\mapsto - \ln x$, the first and second properties follow.
Since $\ln$ is concave, using Jensen's inequality
\begin{align*}
\rho_{\pi_k,h_k}(v) \le&  \frac{-1}{\gamma} \sum_{\theta'\in\mathcal{S}}p_{\theta_k \theta'}\int_{[0,\infty)}  \ln \left( e^{-\gamma v(h_k,\pi_k(h_k), f(\theta_k,\pi_k(h_k),z),\theta')} \right) \nu(dz)\\
=&\sum_{\theta'\in\mathcal{S}} p_{\theta_k \theta'} \int_{[0,\infty)} v(h_k,\pi_k(h_k), f(\theta_k,\pi_k(h_k),z),\theta') \nu(dz)\\
=& E[v(h_k,\pi_k(h_k), f(\theta_k,\pi_k(h_k),\xi_k),\theta_{k+1})\mid h_k].
\end{align*}
Hence the third property holds. For proving the fourth property we do the following. For each $h_k\in H_k$ and $\pi= (\pi_k)_{k\in\mathbb{N}} \in \Pi$, set
$$w_{\pi_k,h_k}(z,\theta'):= w(f(\theta_k,\pi_k(h_k),z),\theta')$$
where $w$ is as in (U2) (F2). Then by the definition of the spaces $B_w(H_{k+1})$ and $B_{w_{\pi_k,h_k}}$, for each $v \in B_w(H_{k+1})$, the map $(z,\theta')\mapsto v(h_k,\pi_k(h_k), f(\theta_k,\pi_k(h_k),z),\theta')$ is  in $B_{w_{\pi_k,h_k}}$ for each $h_k\in H_k$, $\pi= (\pi_k)_{k\in\mathbb{N}} \in \Pi$. Let $F_p: B_{w_{\pi_k,h_k}}\to \mathbb{R}$ be given by
$$
F_p(g):=\frac {1}{-\gamma}\ln \left[\sum_{\theta'\in\mathcal{S}} p_{\theta'} \int_{[0,\infty)} e^{-\gamma g(z,\theta')}\nu(dz) \right]
$$
where $p=(p_1,\ldots, p_N)$ is a probability mass function (pmf) on $\mathcal{S}$, and $g\in B_{w_{\pi_k,h_k}}$.
We wish to establish concavity of the above function $F_p$ using the second order Fr\'echet derivative $D^2F_p$ which is a map from $B_{w_{\pi_k,h_k}}$ to the space of all bilinear forms on $B_{w_{\pi_k,h_k}}\times B_{w_{\pi_k,h_k}}$.

By a direct calculation, $D^2F_p(g)$ along a pair of directions $(\eta, \eta^{\prime})$ is given by
\begin{align*}
D^2F_p(g)(\eta,\eta') =-\gamma &\left(\frac{\sum_{\theta'\in\mathcal{S}} p_{\theta'} \int_{[0,\infty)} e^{-\gamma g(z,\theta')}\eta(z,\theta')\eta'(z,\theta')\nu(dz) }{\sum_{\theta'\in\mathcal{S}} p_{\theta'} \int_{[0,\infty)} e^{-\gamma g(z,\theta')}\nu(dz)}\right.\\
&\left.-\frac{\sum_{\theta'\in\mathcal{S}} p_{\theta'} \int_{[0,\infty)} e^{-\gamma g(z,\theta')} \eta(z,\theta')\nu(dz) \sum_{\theta'\in\mathcal{S}} p_{\theta'} \int_{[0,\infty)} e^{-\gamma g(z,\theta')}\eta'(z,\theta')\nu(dz)}{\left(\sum_{\theta'\in\mathcal{S}} p_{\theta'} \int_{[0,\infty)} e^{-\gamma g(z,\theta')} \nu(dz)\right)^2}\right)
\end{align*}
for any $\eta, \eta' \in B_{w_{\pi_k,h_k}}$. Hence using the Cauchy-Schwartz inequality we get for every $\eta\in B_{w_{\pi_k,h_k}}$
\begin{align*}
D^2F_p(g)(\eta,\eta) = -\gamma&\left[\left(\sum_{\theta'\in\mathcal{S}} p_{\theta'} \int_{[0,\infty)} e^{-\gamma g(z,\theta')}\eta^2(z,\theta')\nu(dz) \right)\left(\sum_{\theta'\in\mathcal{S}} p_{\theta'} \int_{[0,\infty)} e^{-\gamma g(z,\theta')}\nu(dz)\right)\right.\\
&\left.-\left(\sum_{\theta'\in\mathcal{S}} p_{\theta'} \int_{[0,\infty)} e^{-\gamma g(z,\theta')}\eta(z,\theta')\nu(dz)\right)^2\right]\bigg/ \left(\sum_{\theta'\in\mathcal{S}} p_{\theta'} \int_{[0,\infty)} e^{-\gamma g(z,\theta')}\nu(dz)\right)^2\\
\le 0 &
\end{align*}
for all $g\in B_{w_{\pi_k,h_k}}$. Thus $F_p$ is concave on $B_{w_{\pi_k,h_k}}$. Moreover, for each $h_k\in H_k$, $\pi= (\pi_k)_{k\in\mathbb{N}} \in \Pi$ and  $v\in B_w(H_{k+1})$, $\rho_{\pi_k,h_k}(v) = F_{p_{\theta_k}}(v(h_k,\pi_k(h_k), f(\theta_k,\pi_k(h_k),z),\theta'))$, where the pmf $p_{\theta_k}$ is $(p_{\theta_k 1}, \ldots, p_{\theta_k N})$. Hence $v\mapsto \rho_{\pi_k,h_k}(v)$ is concave too. This completes the proof of (iv). The inequality in (v) follows from Theorem \ref{P}(iii) and (F2). Indeed, since $v_{k+1}\in B_w(H_{k+1})$ and $y_k = \pi_k (h_k)$, we have
\begin{eqnarray*}
\nonumber 0\le
\rho_{\pi_k,h_k}(v_{k+1}) &\le& \sum_{\theta'\in\mathcal{S}} p_{\theta_k \theta'} \int_{[0,\infty)} v_{k+1}(h_k,\pi_k(h_k), f(\theta_k,\pi_k(h_k),z),\theta') \nu(dz) \\\nonumber
&\le&
d_{v_{k+1}} \sum_{\theta'\in\mathcal{S}} p_{\theta_k \theta'} \int_{[0,\infty)}w(f(\theta_k,\pi_k(h_k),z),\theta')\nu(dz)\\ \nonumber
&\le&
d_{v_{k+1}}\sup_{y\in[0,x_k]} \sum_{\theta'\in\mathcal{S}} p_{\theta_k \theta'} \int_{[0,\infty)}w(f(\theta_k,y,z),\theta')\nu(dz)\\
&\le& d_{v_{k+1}} \alpha w(x_k,\theta_k) = d_{v_{k+1}} \alpha w_{\theta_k}(x_k),
\end{eqnarray*}
for any $h_k\in H_k$ and  $k\in\mathbb{N}.$ Lastly, for $\mu \in [0,1]$, the assertion (vi) follows by taking $v' =0$ and $\lambda =\mu$ in assertion (iv) and then using (i). To prove the case $\mu\geq 1$ we first apply the assertion (ii) with $v' =0$  and $\lambda = \frac{1}{\mu}$, to get $\rho_{\pi_k,h_k}(v/\mu) - \rho_{\pi_k,h_k}(v)/\mu\ge 0$ for each $v \in B_w(H_{k+1})$. Since, given a $v \in B_w(H_{k+1})$, $v':=\mu v \in B_w(H_{k+1})$, we get $\rho_{\pi_k,h_k}(v'/\mu) - \rho_{\pi_k,h_k}(v')/\mu\ge 0$. Or $\mu \rho_{\pi_k,h_k}(v) - \rho_{\pi_k,h_k}(\mu v)\ge 0$.
\end{proof}

\noindent The following result can be found in \cite{dgl} (Theorem A.19). We provide the proof  here for making this self-contained.
\begin{pr}
\label{cov}
Let $X$ be a real-valued random
variable defined on $(\Omega,{\cal F}, P)$ and let $h$ and $g$ be  non-increasing
real-valued measurable functions. Then,
$$E\{h(X)g(X)\} \ge E\{h(X)\}E\{g(X)\},$$
provided that all expectations exist and are finite.
\end{pr}
\begin{proof}
As both $h$ and $g$ are non-increasing $(h(x)-h(y))(g(x)-g(y)) \ge 0$. Using this we obtain the following using two independent and identically distributed random variables $X$ and $Y$
\begin{eqnarray*}
\lefteqn{E[h(X)g(X)]-Eh(X) Eg(X)}\\
&=&E[h(X)g(X)]-E[h(Y)g(X)]\\
&=&E[(h(X) -h(Y))g(X)]\\
&=&E[1_{\{X>Y\}}(h(X) -h(Y))g(X)] + E[1_{\{X<Y\}}(h(X) -h(Y))g(X)] + E[1_{\{X=Y\}}(h(X) -h(Y))g(X)]\\
&=&E[1_{\{X>Y\}}(h(X) -h(Y))g(X)] + E[1_{\{Y<X\}}(h(Y) -h(X))g(Y)]\\
&=&E[1_{\{X>Y\}}\big((h(X) -h(Y))g(X) + (h(Y) -h(X))g(Y)\big)]\\
&=&E[1_{\{X>Y\}}(h(X) -h(Y))\big(g(X) -g(Y)\big)]\\
&\ge&0.
\end{eqnarray*}
Hence the proposition is true.
\end{proof}
The proof of the next result is a slight modification of  \cite[Lemma 3.1]{kami}  or  \cite[Corollary 11.2.10]{st}.

\begin{lem}\label{lem-concave1}
Let $g: [0,\infty) \to [0,\infty)$ be a  concave function  such that $g(0) \ge 0$. For every $0<y'<y$, we have \begin{equation*}
\frac{g(y')}{y'} \ge \frac{g(y)}{y}.
\end{equation*}
\end{lem}
\begin{proof}
Let $\lambda \in [0,1]$ such that $y' = \lambda y$. Then due to concavity of $g$,
\begin{align*}
g(y') & = g(\lambda y) = g(\lambda y + (1-\lambda)0) \ge \lambda g(y) + (1-\lambda) g(0)  = \frac{y'}{y} g(y) + (1-\lambda) g(0)\\
\textrm{or, } \frac{g(y')}{y'} & \ge  \frac{ g(y)}{y} + \frac{(1-\lambda)}{y'} g(0) \ge  \frac{ g(y)}{y} \quad \textrm{as } g(0) \textrm{ is nonnegative.}
 \end{align*}
\end{proof}
\end{appendix}

\bibliographystyle{abbrv}

\end{document}